\documentclass[a4paper,reqno]{amsart}

\usepackage[utf8]{inputenc}
\usepackage[T1]{fontenc}
\usepackage[english]{babel}

\usepackage{amsmath}
\usepackage{amsfonts}
\usepackage{amssymb}
\usepackage{amsthm}
\usepackage[abbrev]{amsrefs}
\usepackage{mathtools}
\usepackage{enumitem}
\usepackage{cleveref}

\newcommand{\C}{\mathbb{C}}
\newcommand{\N}{\mathbb{N}}
\newcommand{\R}{\mathbb{R}}
\newcommand{\T}{\mathbb{T}}
\newcommand{\Z}{\mathbb{Z}}

\newcommand{\F}{\mathcal{F}}

\newcommand{\set}[1]{\left\{#1\right\}}
\newcommand{\paren}[1]{\left(#1\right)}
\newcommand{\norm}[1]{\|#1\|}

\newcommand{\JBX}[1][x]{\langle#1\rangle}
\newcommand{\1}[1]{\frac{1}{#1}}

\let\d\relax
\DeclareMathOperator{\d}{d \!}

\newtheorem{theorem}{Theorem}
\newtheorem*{definition}{Definition}
\newtheorem{proposition}[theorem]{Proposition}
\newtheorem{lemma}[theorem]{Lemma}
\newtheorem{corollary}[theorem]{Corollary}
\newtheorem{remark}[theorem]{Remark}

\numberwithin{equation}{section}
\numberwithin{theorem}{section}

\author{Joseph Adams}
\address{Heinrich-Heine-Universität Düsseldorf, Mathematisches Institut, Universitätsstraße~1, 40225 Düsseldorf, Germany}
\email{joseph.adams@hhu.de}

\makeatletter
\@namedef{subjclassname@2020}{\textup{2020} Mathematics Subject Classification}
\makeatother
\subjclass[2020]{Primary: 35Q55. Secondary: 35Q41}
\keywords{dNLS hierarchy equations --- low regularity local well-posedness --- Fourier restriction norm method}

\title{Well-posedness for the dNLS hierarchy}
\begin{document}
    \begin{abstract}
    We prove well-posedness for higher-order equations in the so-called dNLS hierarchy (also known as part of the Kaup-Newell hierarchy) in almost critical Fourier-Lebesgue and in modulation spaces.
    Leaning in on estimates proven by the author in a previous instalment~\cite{Adams2024}, where a similar well-posedness theory was developed for the equations of the NLS hierarchy, we show the $j$th equation in the dNLS hierarchy is locally well-posed for initial data in $\hat H^s_r(\R)$ for $s \ge \1{2} + \frac{j-1}{r'}$ and $1 < r \le 2$ and also in $M^s_{2, p}(\R)$ for $s \ge \frac{j}{2}$ and $2 \le p < \infty$. Supplementing our results with corresponding ill-posedness results in Fourier-Lebesgue and modulation spaces shows optimality.

    Our arguments are based on the Fourier restriction norm method in Bourgain spaces adapted to our data spaces  and the gauge-transformation commonly associated with the dNLS equation. For the latter we establish bi-Lipschitz continuity between appropriate modulation spaces and that even for higher-order equations `bad' cubic nonlinear terms are lifted from the equation.
    \end{abstract}
	\maketitle
    \tableofcontents

    \clearpage


\section{Introduction}\label{sec:introduction}

The derivative nonlinear Schrödinger (dNLS) equation
\begin{equation}\label{eq:dnls}
    \begin{cases}
    i\partial_t u + \partial_x^2 u = i\partial_x (|u|^2 u)\\
    u(t = 0) = u_0
    \end{cases}
\end{equation}
with initial data $u_0$, is a canonical object of study in the field of well-posedness theory for dispersive PDE. It arises as a model in various branches of physics, ranging from the propagation of circularly polarized Alfvén waves in magnetized plasma to the propagation of ultra-short pulses in optical fibers. We direct the interested reader to~\cites{MR462141,agrawal2000nonlinear,PhysRevA.27.1393,PhysRevA.23.1266} for an overview of its origins.

Its analysis, in the sense of low-regularity well-posedness, compared with its closely related cousin, the (de)focusing cubic nonlinear Schrödinger (NLS) equation
\begin{equation}\label{eq:nls}
    i\partial_t u + \partial_x^2 u = \pm 2|u|^2 u,
\end{equation}
is considered to be strictly more difficult, because of the additional derivative in the nonlinearity. In particular, one of the nonlinear terms $|u|^2\partial_x u$ in~\eqref{eq:dnls} is much less well behaved than the remaining term $u^2\partial_x \overline{u}$.

One way to absolve the equation of this `issue' and still be able to achieve well-posedness within the framework of the Fourier restriction norm method, or more generally by fixed-point arguments, is by utilising the gauge-trans\-for\-ma\-tion
\begin{equation}\label{eq:gauge-transformation}
    u(x, t) \mapsto v(x, t) \coloneqq \exp\left(-i \int_{-\infty}^{x} |u(y, t)|^2 \d{y}\right) u(x, t)
\end{equation}
which removes the ill-behaved $|u|^2\partial_x u$ by translating~\eqref{eq:dnls} to the equation
\begin{equation}\label{eq:gt-dnls}
    i\partial_t v + \partial_x^2 v = -i v^2 \partial_x \overline{v} - \1{2} |v|^4 v
\end{equation}
for an unknown function $v$. (The initial value is also adapted in an appropriate fashion.) The continuity properties of the gauge-transformation then ensure essentially\footnote{Using the gauge-transformation muddies the uniqueness properties of the solution. See Remark~\ref{remark:uniqueness} where this issue is further discussed.} the equivalence of Cauchy problems associated with both~\eqref{eq:dnls} and~\eqref{eq:gt-dnls}. See~\cites{AG2005,HayashiOzawa1992,Guo2021,Takaoka1999} and the references therein, where this approach has successfully been applied in a variety of function spaces.

Though even after transformation, solely using energy or smoothing estimates does not suffice to prove (near optimal) local well-posedness results. As was layed out in~\cite{AG2005}, for certain frequency constellations one is forced to exploit the resonance relation to eke out a fraction (in the $L^2$-based setting) of a derivative in order to close a contraction argument. So there is certainly some added complexity when dealing with the dNLS equation in comparison to the NLS equation.

Furthermore, the dNLS equation is a completely integrable system, which entails but is not limited to possessing an infinite hierarchy of conserved quantities and being induced by (one of) the first of these quantities. Subsequent equations may be induced in a similar fashion to produce what we refer to in the title of this paper as \emph{the dNLS hierarchy}. As the NLS equation is also completely integrable, one can analogously look at an NLS hierarchy. (How these conserved quantities are derived for dNLS and what is meant by `induce' will be made more precise in Section~\ref{sec:dnls-derivation}.)

Grounded in the recently published paper~\cite{Adams2024} by the author, in which the well-posedness theory of the NLS hierarchy is studied, the natural question arises what a similar theory would look like for the dNLS hierarchy, keeping in mind its added complexities?

Goal of the present paper is to (at least partially) answer this question. More precisely we will be proving low-regularity well-posedness results for a general class of PDE, encompassing all equations in the dNLS hierarchy, in classical Sobolev spaces $H^s(\R)$, Fourier-Lebesgue spaces $\hat H^s_r(\R)$ (sometimes written as $\F L^{s,r'}(\R)$ in the literature) and modulation spaces $M^s_{2,p}(\R)$ defined by the norms
\begin{equation}
    \norm{u}_{\hat H^s_r} = \norm{u}_{\F L^{s,r'}} = \norm{\JBX[\xi]^s \hat u}_{L^{r'}} \quad\text{and}\quad \norm{u}_{M^s_{2,p}} = \norm{\JBX[n]^s \norm{\Box_n u}_{L^2}}_{\ell^p_n(\Z)}
\end{equation}
respectively, with a family of isometric decomposition operators $(\Box_n)_{n\in\Z}$. We refer to the author's previous work~\cite{Adams2024}*{Section~1.2} for precise definitions and an overview of properties, i.e. embeddings, interpolation and duality theory of these function spaces.

While of course we embrace the integrability structure of the dNLS hierarchy equations for their derivation, we will not be making use of it for proving our well-posedness results. Rather we welcome the fact that our techniques enable us to prove well-posedness for a much larger class of PDE (that nevertheless includes the dNLS hierarchy equations), due to their robustness towards changes in the PDE that lead to them no longer being completely integrable.

The techniques we will be using to argue well-posedness are the Fourier restriction norm method in appropriate Bourgain spaces $X_{s,b}$ adapted to our data spaces, together with bilinear refinements of Strichartz estimates. We will also be heavily leaning in on the estimates proven for the NLS hierarchy equations in~\cite{Adams2024} by the author and general smoothing estimates of Kato type. As a convenience we recall all necessary estimates in Section~\ref{sec:known-estimates}.

\subsection{Notation and function space properties}

As the present paper may be viewed as a continuation or extension of the author's previous work on the NLS hierarchy, we will refrain from (re)defining our notational conventions and instead refer the reader to~\cite{Adams2024}*{Section~1.2} for reference on such matters.

In addition, we will be using some estimates for modulation spaces not yet given in~\cite{Adams2024} so we will use this opportunity to cite these from the literature.
Of particular use will be a Sobolev-type embedding adapted to modulation spaces, a proof of which may be found in~\cite{Chaichenets2018}*{Prop.~2.31}: Let $s_1, s_2 \in\R$ and $1 \le p,q_1,q_2 \le \infty$ then
\begin{equation}\label{eq:modulation-sobolev}
\norm{f}_{M^{s_1}_{p,q_1}(\R^n)} \lesssim \norm{f}_{M^{s_2}_{p,q_2}(\R^n)} \quad\text{if and only if}\quad s_1 - s_2 > \frac{n}{q_2} - \frac{n}{q_1} > 0.
\end{equation}

The other estimates we will be needing are all with regard to multiplication of modulation space functions. We start by mentioning the well known fact that $M_{\infty, 1}$ is a Banach-Algebra, see~\cite{Chaichenets2018}*{Prop.~4.2}. In fact, as is also mentioned after that Proposition, since $M^s_{p,q}(\R^n)$ continuously embeds into $M_{\infty, 1}(\R^n)$, if $q=1$ and $s \ge 0$, or if $q > 1$ and $s > \frac{n}{q'}$, we know $M^s_{p,q}(\R^n)$ also to be an algebra in those cases.

More generally we have a form of generalised Leibniz rule for modulation spaces: Let $s \ge 0$ and $1 \le p, p_1, p_2, \tilde{p}_q, \tilde{p}_2, q, q_1, q_2, \tilde{q}_1, \tilde{q}_2 \le \infty$, such that $\1{p} = \1{p_1} + \1{p_2} = \1{\tilde p_1} + \1{\tilde p_2}$ and $\1{q'} = \1{q_1'} + \1{q_2'} = \1{\tilde q_1'} + \1{\tilde q_2'}$, then
\begin{equation}\label{eq:modulation-leibniz}
\norm{fg}_{M^s_{p,q}(\R^n)} \lesssim \norm{f}_{M^s_{p_1, q_1}(\R^n)} \norm{g}_{M_{p_2,q_2}(\R^n)} + \norm{f}_{M_{\tilde{p}_1, \tilde{q}_1}(\R^n)} \norm{g}_{M^s_{\tilde{p}_2, \tilde{q}_2}(\R^n)}.
\end{equation}
Taking the uniform-decomposition definition of modulation spaces as known, as simple proof is as follows: We rewrite $\Box_m (f g)$ as $\sum_{k+\ell=m} (\Box_k f) (\Box_\ell g)$ using knowledge of the support of convolutions.
\begin{align*}
	\norm{fg}_{M^s_{p,q}} &= \norm{ \norm{\JBX[m]^s \Box_m (f g)}_{L^p} }_{\ell^q_m(\Z)} \lesssim \norm{ \norm{\JBX[m]^s \sum_{k+\ell=m} (\Box_k f) (\Box_\ell g)}_{L^p} }_{\ell^q_m(\Z)}\\
	&\lesssim \norm{ \sum_{k+\ell=m} (\JBX[k]^s + \JBX[\ell]^s) \norm{(\Box_k f) (\Box_\ell g)}_{L^p} }_{\ell^q_m(\Z)}
	\intertext{After applying the triangle inequality $\JBX[m]^s \lesssim \JBX[k]^s+\JBX[\ell]^s$ we use Hölder's inequality depending on which weight is present. Finishing the proof with applications of the triangle and Young's inequality we arrive at the desired upper bound.}
	&\lesssim \norm{ \sum_{k+\ell=m} \JBX[k]^s \norm{\Box_k f}_{L^{p_1}} \norm{\Box_\ell g}_{L^p_2}}_{\ell^q_m(\Z)} + \norm{\JBX[\ell]^s \norm{\Box_k f}_{L^{\tilde{p}_1}} \norm{\Box_\ell g}_{L^{\tilde{p}_2}}}_{\ell^q_m(\Z)}\\
	&\lesssim \norm{f}_{M^s_{p_1, q_1}} \norm{g}_{M^0_{p_2,q_2}} + \norm{f}_{M^0_{\tilde{p}_1,\tilde{q}_1}} \norm{g}_{M^s_{\tilde{p}_2,\tilde{q}_2}}.
\end{align*}
For further properties of modulation spaces we recommend consulting~\cites{Benyi2020,Chaichenets2018}.

In addition we will be using the classic Gagliardo-Nirenberg inequality in deriving a-priori bounds for the dNLS hierarchy equations. We take advantage of the phrasing from~\cite{GagliardoNirenberg}: Let $1 \le r,p,q \le \infty$, $\ell \in \N_0$, $k \in \N$ and $\frac{\ell}{k} \le \theta \le 1$ such that
\begin{equation}
    \1{r} - \frac{\ell}{n} = \theta\paren{\1{p}-\frac{k}{n}} + (1-\theta)\1{q}
\end{equation}
holds. Then one has the inequality
\begin{equation}\label{eq:gagliardo-nirenberg}
    \norm{\nabla^\ell f}_{L^r(\R^n)} \lesssim \norm{\nabla^k f}_{L^p(\R^n)}^\theta \norm{f}_{L^q(\R^n)}^{1-\theta}
\end{equation}
under the additional constraints that $\theta < 1$ if $r = \infty$ and $1 < p < \infty$; or $f$ is vanishing at infinity if $q = \infty$, $k < \frac{n}{p}$ and $\ell = 0$.

\subsection{Organisation of the paper}

In Section~\ref{sec:dnls-derivation} we will be deriving and defining what is referred to in the title of this paper as the dNLS hierarchy. We will also review what is known about the gauge-transformation associated with the dNLS equation. In addition we will prove its continuity as a map between appropriate modulation spaces and argue that applied to the higher-order dNLS hierarchy equations it also leads to more well-behaved models. We will be referring to these more well-behaved models as gauged dNLS equations and make reference to them in our well-posedness theorems.

Then in Section~\ref{sec:statement-of-results} we quickly review prior work associated with (higher-order) dNLS equations before stating our main results, followed by a discussion of the latter.

Moving towards proofs of the theorems, in Section~\ref{sec:known-estimates}, we give an overview of the linear and multilinear estimates from~\cite{Adams2024} that we will be using to argue well-posedness for higher-order dNLS hierarchy/gauged dNLS equations, for the reader's convenience. In addition we will be making use of an estimate for the resonance relation which we take from the literature.

The proofs for Theorems~\ref{thm:gauged-wellposed-hat} and~\ref{thm:gauged-wellposed-modulation} are contained in Section~\ref{sec:multilinear-estimates}, where first we deal with estimates regarding well-posedness in Fourier-Lebesgue spaces, followed by the same for modulation spaces. The Theorems~\ref{thm:hierarchy-wellposed-hat-spaces} and~\ref{thm:hierarchy-wellposed-modulation-spaces} regarding well-posedness of the dNLS hierarchy equations themselves follow from the former and use of the gauge-transformation.

In Section~\ref{sec:illposedness} we give proofs of our ill-posedness results associated with higher-order dNLS equations. These show that our well-posedness results are optimal (up to the endpoint) and that within the framework of techniques we are using, no lower threshold of initial regularity of the data is possible, while still achieving local well-posedness results.

To wrap up, in Appendix~\ref{sec:appendix} we list the first few equations of the dNLS hierarchy together with their gauge-transformed variants where appropriate. This shall serve as a point of reference and give the interested reader an overview of what typical nonlinearities in the hierarchy look like.

\subsection*{Acknowledgements}
This work is part of the author's PhD thesis. He would like to greatly thank his advisor, Axel Grünrock, for suggesting this line of problems and his continued and ongoing support.


\section{Description of the dNLS hierarchy}\label{sec:dnls-derivation}

Keeping in line with the literature we referenced in~\cite{Adams2024} describing the derivation of the NLS hierarchy equations, we stick to~\cites{Alberty1982-i, Sasaki1982-ii} for the dNLS hierarchy equations. For literature dealing more generally with completely integrable systems we recommend the reader consult~\cites{Faddeev1987, Palais1997} and references therein.

In the forthcoming subsections we describe how dNLS and associated higher-order equations arise as a compatibility condition for a linear scattering problem and how these equations are amenable to being recast in a more well-behaved class using the gauge-transformation~\eqref{eq:gauge-transformation}. We will also touch on why this transformation leaves the well-posedness question (mostly) intact, specifically we are referring to the regularity of the gauge-transformation itself.

\subsection{Deriving dNLS hierarchy equations}

The general setting we start out in is a linear scattering problem~\cite{Alberty1982-i}*{eq.~(1.1)} of the form
\begin{equation}\label{eq:general-linear-scattering}
    \d{v} = \Omega v
\end{equation}
involving an $N \times N$ matrix of differential one-forms $\Omega$ depending on a spectral parameter $\zeta \in \C$. Its zero-curvature (also called integrability) condition~\cite{Alberty1982-i}*{eq.~(1.2)}\cite{Sasaki1982-ii}*{eq.~(2.3)} reads
\begin{equation}\label{eq:zcc}
    0 = \d{\Omega} - \Omega \wedge \Omega
\end{equation}
and, for appropriate choice of $\Omega$, leads to various well-known nonlinear evolution equations. Choosing the right Ansatz for $\Omega$ decides which particular set of equations one manages to derive. In~\cite{Adams2024} and~\cite{Alberty1982-i} the Ansatz $\Omega = (\zeta R_0 + P)\d{x} + Q(\zeta)\d{t}$, where the $\d{x}$ part of $\Omega$ depends only linearly on the spectral parameter $\zeta \in \C$, was chosen. One picks the involved matrices as
\begin{equation}
    R_0 = \begin{pmatrix}
    -i & 0 \\ 0 & i
    \end{pmatrix}
    \quad\text{and}\quad
    P = \begin{pmatrix}
    0 & q \\ r & 0
    \end{pmatrix},
\end{equation}
where we leave $Q$ open for the time being. The entries $q$ and $r$ (which are functions depending on $x$ and $t$) are referred to as potentials along which the scattering in~\eqref{eq:general-linear-scattering} happens.

This Ansatz leads to (for example) the NLS and (m)KdV hierarchies of equations\footnote{The astute reader will note, that both the NLS and mKdV equations are embedded within the Ablowitz-Kaup-Newell-Segur (AKNS) hierarchy, a name more commonly used in the inverse scattering community literature, see for example~\cites{AKNS1974,Faddeev1987}.}, depending again on the particular choice of relation between the two potentials $q$ and $r$ and matrix $Q$. In order to derive the dNLS hierarchy equations we follow~\cite{Sasaki1982-ii}*{eq.~(2.4)} and now instead choose $\Omega = (\zeta^2 R_0 + \zeta P)\d{x} + Q(\zeta)\d{t}$ with the same matrices $R_0$ and $P$ as previously, again leaving $Q$ unspecified for now.

A prolonged calculation that we will not reproduce for brevity's sake then shows that the compatibility condition~\eqref{eq:zcc} has an equivalent formulation as a Hamiltonian equation for our two potentials $q$ and $r$
\begin{equation}\label{eq:general-hamiltonian}
    \frac{\d}{\d{t}}u = J \frac{\delta}{\delta u} \mathcal{H},
\end{equation}
see~\cite{Sasaki1982-ii}*{eq.~(4.11)}. In this equation $u = \left(\begin{smallmatrix}r\\q\end{smallmatrix}\right)$ is a vector containing our potentials and $J = -2i \left(\begin{smallmatrix}0 & 1 \\1 & 0\end{smallmatrix}\right)\partial_x$ is an operator (different from the one involved in the derivation of the NLS hierarchy, cf.~\cite{Adams2024}*{eq.~(2.3)}). What is left is to define the Hamiltonian $\mathcal{H}$ that is namesake to~\eqref{eq:general-hamiltonian}.

The Hamiltonian $\mathcal{H}$ has a strikingly similar form as for the NLS hierarchy equations
\begin{equation}\label{eq:hamiltonian-dfn}
    \mathcal{H} = \sum_{n=0}^\infty \alpha_n(t) I_n,
\end{equation}
see~\cite{Adams2024}*{eq.~(2.4)} for comparison. The $\alpha_n(t)$ are derived from the choice of $Q$ we left open previously, and the $I_n$ are conserved quantities of the equations in the dNLS hierarchy, in particular dNLS itself. Appropriate choices of the $\alpha_n(t)$ will thus yield the dNLS hierarchy equations, for which (individually) the $I_n$ are the Hamiltonians.

Last thing is to state the individual Hamiltonian $I_n$: In~\cite{Sasaki1982-ii}*{eqns.~(3.3) and~(3.4)} we are given explicit expressions for deriving these conserved quantities/Ham\-il\-to\-ni\-ans recursively
\begin{equation}\label{eq:hamiltonian-recusion}
    I_n = \int_\R q Y_n \d{x} \;\;\;\text{and}\;\;\; Y_{n+1} = \1{2i}\left[ \partial_x Y_n + q \sum_{k=0}^{n} Y_{n-k}Y_k \right] \;\;\; \text{with} \;\;\; Y_0 = - \frac{r}{2i}.
\end{equation}
The resemblance between~\eqref{eq:hamiltonian-recusion} and~\cite{Adams2024}*{eq.~(2.5)} is undeniable, though the discerning reader will note that the initial condition for this recursion is different, as well as the sum going up to $k = n$ (rather than $k = n-1$).

For later reference we would like to give a lemma describing elementary properties of the $Y_n$ all of which may be verified by a simple inductive argument, so we omit the proof.
\begin{lemma}\label{lem:Yn-properties}
    For $n\in\N$ the terms $Y_n$ have the following properties:
    \begin{enumerate}
    \item $Y_n$ is a sum of monomials in $q$, $r$ and their derivatives.
    \item $Y_n$ as a polynomial is of homogeneous order, where we define the order of a monomial to be sum of twice the total number of derivatives and the number of factors in it. The order of any monomial in $Y_n$ is $2n+1$.
    \item Every monomial in $Y_n$ has a total number of factors $r$, or its derivatives, one greater than the total number of factors $q$, or its derivatives.
    \item The coefficients of the monomials in $Y_n$ are a positive integer multiples of $(-1)^k(2i)^{k-2n-1}$, where $k$ is the total number of derivatives in a given monomial.
    \item $Y_n$ has a single term that consists of just one factor, it is $-(2i)^{-n}\partial_x^{n}r$.
    \end{enumerate}
\end{lemma}

We are now ready to give the definition, i.e. fix a choice of coefficients $\alpha_n$ in~\eqref{eq:hamiltonian-dfn}, of what is referred to in the title of this paper as the dNLS hierarchy.

\begin{definition}
    For $j \in\N$ we define the $j$th dNLS hierarchy equation to be the Hamiltonian equation for the potential $q(x, t)$ in~\eqref{eq:general-hamiltonian}, where we choose $\alpha_{2j-1} = 2^{2j-1}$ and $\alpha_n = 0$ for  $n \not= 2j-1$ in~\eqref{eq:hamiltonian-dfn}. We identify occurrences of the potential $r(x, t)$ with the complex conjugate of $q(x, t)$, i.e. $r = +\overline{q}$.
\end{definition}

Having defined what we deem to be the dNLS hierarchy equations we may quickly establish an equivalent theorem to~\cite{Adams2024}*{Theorem~2.3} that describes the general form of such an equation. We leave its proof to the reader as it differs only in details from the one in~\cite{Adams2024}.

Note that this is also the point where we switch back to the more common notation of calling the unknown function $u$ (instead of $q$ or $r$). This is not to be confused with the vector of potentials $u = \left(\begin{smallmatrix}r\\q\end{smallmatrix}\right)$ used in~\eqref{eq:general-hamiltonian}.

\begin{theorem}
    For $j \in\N$ there exist coefficients $c_{k, \alpha} \in \Z+i\Z$ for every $\alpha \in \N_0^{2k+1}$ with $|\alpha| = 2j-k-1$, for $1 \le k \le 2j-1$, such that the $j$th dNLS hierarchy equation may be written as
    \begin{equation}\label{eq:dnls-hierarchy}
        i\partial_t u + (-1)^{j+1} \partial_x^{2j} u = \sum_{k=1}^{2j-1} \sum_{\substack{\alpha\in\N_0^{2k+1}\\|\alpha| = 2j-k-1}} c_{k,\alpha} \partial_x \left( \partial_x^{\alpha_1} u \prod_{\ell = 1}^k \partial_x^{\alpha_{2\ell}}\overline{u}\partial_x^{\alpha_{2\ell+1}}u \right).
    \end{equation}
\end{theorem}

\begin{remark}
    We give some points of interest and remarks:
    \begin{enumerate}[wide]
        \item Breaking the definitions down in order to better uncover the structure of the dNLS hierarchy equations, we note that for $n = 2j-1$ the $j$th dNLS equation is given by
        \begin{equation}\label{eq:abstract-dnls}
            i\partial_t u = 2\alpha_{n} \partial_x \frac{\delta}{\delta \overline{u}} \int_\R u Y_{n} \d{x}.
        \end{equation}
        \item The main difference between the equations of the NLS and dNLS hierarchies is that the latter has an additional derivative on each nonlinear term. This is what makes its analysis more difficult, as the nonlinear term $|u|^2\partial_x u$ and its higher-order variants (where none of the derivatives fall on the complex conjugated factor $\overline{u}$) are quite ill-behaved. This is the reason we will be using the gauge-transformation, on which we will give more details in the next subsection.
        \item The first dNLS hierarchy equation ($j=1$) corresponds to the classical dNLS equation~\eqref{eq:dnls}. The higher-order equations, beyond the dNLS and fourth-order ($j=2$) equation, do not, to the author's best knowledge, appear in the literature. We list the first few equations of the hierarchy in Appendix~\ref{sec:appendix}. A further (interleaving) sequence of higher-order PDEs (with odd order of dispersion) can be defined and corresponds to non-zero choices of $\alpha_n$, for $n \not= 2j-1, j\in\N$. We list these in the same appendix. 
        \item Choosing the opposing sign convention $r = -\overline{q}$ also leads to a hierarchy of dNLS-like equations. As, in contrast to NLS, there is no meaningful difference between a focusing or defocusing case depending on the sign in front of the nonlinearity, our sign choice is of no significant importance. We fix it merely to have a designated convention for the name and choose to stay in line with the dNLS equation already present in the literature.
        \item Figuring out a non-recursive description of the coefficients involved in the dNLS hierarchy (or even determining, beyond~\eqref{eq:dnls-hierarchy}, which nonlinear terms appear at all) is, to the author's best knowledge and in general, an unsolved problem. Such further insight into the nonlinearities may in the future aid phrasing well-posedness results dependent on a non-resonance condition (only fulfilled by the actual hierarchy equations).

        In the following subsection, where we explore the action of the gauge-trans\-for\-ma\-tion on the dNLS hierarchy equations, we will at least be able to obtain the coefficients of `bad' cubic nonlinear terms, where no derivatives fall on the complex conjugated factor $\overline{u}$. These `bad' cubic terms are the higher-order generalisations of $|u|^2\partial_x u$ from the nonlinearity of dNLS.
        \item Choosing non-zero values for the even numbered coefficients $\alpha_{2j}$ (and zero for all others) leads to a set of equations that have the same linear parts as the equations in the mKdV hierarchy (see Appendix~\ref{sec:appendix}). It seems these do not appear independently in the literature, but would surely also make for an interesting object of study. Though we do not pursue this in this work.
    \end{enumerate}
\end{remark}

\begin{remark}\label{rem:criticality}
    Now is the right place to establish the critical regularity $s_c(j, r)$ of the dNLS hierarchy equations in Sobolev and more generally Fourier-Lebesgue spaces\footnote{Modulation spaces are not well-behaved under transformations of scale, due to the uniform frequency decomposition involved, thus there is no proper notion of criticality.}.

    In a similar fashion to dNLS itself, the higher-order dNLS hierarchy equations are also invariant under the transformation of scale $u_\lambda(x, t) = \lambda^\1{2}u(\lambda x, \lambda^{2j}t)$, meaning if $u$ is a solution of a dNLS-like equation with initial data $u_0$, then so is $u_\lambda$ with initial data $u_{0,\lambda}(x) = \lambda^\1{2}u_0(\lambda x)$.

    This leads to the critical regularity being $s_c(j, r) = \1{r} - \1{2}$, i.e. the $L^2$-norm stays invariant under this transformation on the scale of Sobolev spaces and $\hat H^\1{2}_1$ on the scale of Fourier-Lebesgue spaces for $r \to 1$.
    Our determined goal is to establish well-posedness of the dNLS hierarchy equations in spaces that are very close to these critical spaces.
\end{remark}

\subsection{The gauge-transformation}\label{sec:gauge-transformation}

As is mentioned above there are certain nonlinear terms that appear in the dNLS hierarchy equations that are gravely less well-behaved than their fellows. These are terms like $|u|^2\partial_x u$ from~\eqref{eq:dnls}, where all derivatives that lie on a cubic nonlinear term fall onto one of the factors that is not the complex conjugate of the unknown solution $u$. As the reader may verify in Appendix~\ref{sec:appendix} these types of nonlinear terms do in fact crop up in the higher-order equations too.

Before we move on to proving well-posedness results for the dNLS hierarchy equations we must first absolve ourselves of these ill-behaved nonlinear terms. To do this we will be making use of the gauge-transformation that is already a well-known tool in the context of the dNLS equation itself:
\begin{equation}\label{eq:chapter-gauge-transformation}
    \mathcal{G}_\pm : u(x, t) \mapsto v(x, t) \coloneqq \exp\left(\pm i \int_{-\infty}^{x} |u(y, t)|^2 \d{y}\right) u(x, t).
\end{equation}
See~\cites{AG2005,HayashiOzawa1992,Guo2021,Takaoka1999}, for example.

For the dNLS equation the gauge-transformation~\eqref{eq:chapter-gauge-transformation} is useful in the following sense: given a function $u$, it solves the dNLS equation~\eqref{eq:dnls} if and only if $v(x, t) := \mathcal{G}_-(u)(x, t)$ solves the gauge-transformed dNLS equation~\eqref{eq:gt-dnls}. Vice versa when you apply the gauge-transformation's inverse $\mathcal{G}_+$.

We want to explore how the gauge-transformation can help us in a similar way in order to simplify, or even enable, the well-posedness analysis of higher-order dNLS hierarchy equations. For this we must first find the right notion of `simpler' equation, which is specific enough in order for us to be able to achieve well-posedness results for and also general enough so that it is a superset of the image of the dNLS hierarchy equations under the gauge-transformation. We find the following definition appropriate.
\begin{definition}
    For $j\in\N$ we call a PDE a ($j$th order) gauged dNLS equation, if there exist coefficients $c_{k,\alpha} \in \C$, for $1 \le k \le 2j$, and $\alpha \in \N_0^{2k+1}$ with $|\alpha| = 2j - k$, such that $c_{1,\alpha} = 0$ if $\alpha_2 = 0$ and the PDE may be written as
    \begin{equation}\label{eq:gauged-dnls-eqn}
        i\partial_t v + (-1)^{j+1}\partial_x^{2j} v = \sum_{k=1}^{2j}\sum_{\substack{\alpha\in\N_0^{2k+1}\\|\alpha|=2j-k}} c_{k,\alpha} \partial_x^{\alpha_1}u \prod_{i=1}^{k} \partial_x^{\alpha_{2i}}\overline{u} \partial_x^{\alpha_{2i+1}}u.
    \end{equation}
\end{definition}
The difference between dNLS hierarchy equations and gauged dNLS equations, in their general form, is evidently rather small. The linear parts of the equations coincide for one. Regarding the cubic nonlinear terms, the gauged dNLS equations cannot contain so called `bad' cubic terms that have none of their derivatives fall on the factor $\overline{u}$ in the cubic. This is exactly the advantage the gauge-transformation delivers. With regard to the higher-order nonlinear terms, the small price we have to pay for the elimination of the `bad' cubic terms is that we incur an additional term of the form $|u|^{2j}u$, without any derivatives lying on it.

\begin{remark}
    We point out that transitioning from dNLS hierarchy equations to gauged dNLS equations does not change the notion of criticality, that was investigated in Remark~\ref{rem:criticality}. This is because we are, at most, leaving a cubic nonlinear term away and are gaining a term of the form $|u|^{2j}u$, that is invariant with respect to the same transformation of scale.
\end{remark}

Our goal for the rest of this subsection will be to establish, that the gauge-transformation does indeed translate between the dNLS hierarchy equations and what we are now referring to as gauged dNLS equations. This will then later allow us, conditioned on the continuity of the gauge-transformation, to prove well-posedness solely for gauged dNLS equations and pull-back these results to the actual equations of interest: the dNLS hierarchy equations. In this spirit we will be proving the following proposition.
\begin{proposition}\label{prop:gauge-transformation}
    Let $j\ge2$, $u(x, t)$ be a function and $v(x, t) := \mathcal{G}_-(u)(x,t)$ its gauge-transform. Then $u$ solves the $j$th order dNLS hierarchy equation if and only if $v$ solves a (corresponding) gauged dNLS equation. And vice versa for the inverse transformation $\mathcal{G}_+$.
\end{proposition}
Even though this proposition does not exactly specify \emph{which} gauged dNLS equation $v$ would solve, this proposition is sufficient for our purposes, since our well-posedness theorems are so general as to cover the whole class of gauged dNLS equations.

Relating to proof strategy, we will be investigating the coefficients of the `bad' cubic nonlinear terms in the dNLS hierarchy equations and show that these coincide with those coefficients of `bad' cubic terms that are lifted when one uses the gauge-transformation.
We point out that this makes the dNLS hierarchy equations natural, beyond being derived from a completely integrable system, in the sense that their coefficients for `bad' cubic terms are the unique\footnote{This uniqueness is only up to scaling of the coefficients. What is actually unique is the relationship (quotient) between the coefficients.} set that are amenable to use of the gauge-transformation.

As was also the case for the NLS hierarchy equations in~\cite{Adams2024}, there is no specific understanding of the coefficients or finer structure of nonlinearities for the higher-order dNLS hierarchy equations present in the literature, to the author's best knowledge. So the following proposition, where the coefficients of at least the `bad' cubic terms are uncovered, is a first.
\begin{proposition}\label{prop:dnls-coeff}
    Let $n \ge 1$. For $0 \le k \le n$ the coefficient of the cubic nonlinear term $(\partial_x^{n-k}u) \overline{u} (\partial_x^k u)$ is equal to
    \begin{equation}\label{eq:dnls-coeff}
        \frac{4(-1)^{n+1}\alpha_n}{(2i)^{n+2}} \paren{\binom{n+2}{k+1} - \delta_{0, k} - \delta_{n, k}},
    \end{equation}
    where $\delta_{a, b}$ is the Kronecker delta.
\end{proposition}
For $n=1$ is an easy and well-known result: the coefficient of $|u|^2u_x$ in the dNLS equation is $2i$. We note that there is some level of redundancy in the statement as the terms $(\partial_x^{n-k}u) \overline{u} (\partial_x^k u)$ and $(\partial_x^{k}u) \overline{u} (\partial_x^{n-k} u)$ are the same by commutativity. This representation also still contains a choice of coefficients $\alpha_n$. For the dNLS hierarchy we have made this choice, which seems canonical in relation to the coefficients appearing in the gauge-transformation, see Lemma~\ref{lem:gauge-coeff}.

\begin{remark}
    Figuring out the coefficients of the cubic nonlinear terms in general or of any of the higher-order terms also seems an interesting problem. Though the author finds that more delicate methods must be required in order to uncover these, as there is less of an obvious pattern compared with the `bad' cubics.
\end{remark}

\begin{proof}[Proof of Proposition~\ref{prop:dnls-coeff}]
We will prove the claim for $n \ge 2$ only, to eliminate some edge-cases. Referring to~\eqref{eq:abstract-dnls}, which we now understand for general $n\in\N$, we must ask ourselves: where do the `bad' cubic terms come from?\footnote{Even though we haven't formally defined what `bad' cubic terms for mKdV-like equations with an extra derivative are (so where $n$ is even), we will deal with them to be analogues of those for the dNLS hierarchy equations. That is where none of the derivatives in a cubic term fall on $\overline{u}$.}

Working our way backwards, such `bad' terms, say $(\partial_x^{n-k}u) \overline{u} (\partial_x^k u)$, for $0 \le k \le n$, originate (before applying the derivative $\partial_x$ present in~\eqref{eq:abstract-dnls}) from cubic terms in $\frac{\delta}{\delta \overline{u}} \int_\R u Y_{n}\d{x}$ that also have no derivatives lying on $\overline{u}$ and a single derivative fewer in total, for example $(\partial_x^{n-1-k}u) \overline{u} (\partial_x^k u)$, for $0 \le k \le n-1$.

Recursing again, past the functional derivative, such cubic terms with $n-1$ total derivatives, but none on $\overline{u}$, can only originate from quartic terms in the integrand of the Hamiltonian, where at least one of the two $\overline{u}$ factors has no derivatives lying upon it. In turn, since we are multiplying with $u$ in the integral, these come from cubic terms in $Y_{n}$ where at least one of the two factors $\overline{u}$ has no derivatives lying upon it. General form of these terms is then $(\partial_x^{n-1-k}u) \overline{u} (\partial_x^k \overline{u})$, for $0 \le k \le n-1$.

To ease notation let $K_n(k)$ refer to the coefficient of $(\partial_x^{n-1-k}u) \overline{u} (\partial_x^k \overline{u})$ in $Y_n$, for $0 \le k \le n-1$. From here on out we will also use the convention $c = \1{2i}$, as this factor will appear often.

Our initial task is now to determine $K_n(k)$, for $n \ge 1$ and $0 \le k \le n-1$.
Looking at the recursive definition of $Y_{n+1}$ in~\eqref{eq:hamiltonian-recusion}
\begin{equation}
    Y_{n+1} = c\left[ \partial_x Y_n + u \sum_{k=0}^{n} Y_{n-k}Y_k \right] \;\;\; \text{with} \;\;\; Y_0 = - c\overline{u},
\end{equation}
we can determine that cubic terms with coefficients $K_{n+1}(k)$ appear in $Y_{n+1}$ in two ways:
\begin{enumerate}
    \item from the first summand in the brackets, if a term in $Y_n$ that also has a factor $\overline{u}$ with no derivatives gets differentiated, by Leibniz' rule,
    \item in the sum, since the whole sum is multiplied with $u$, if for either $k=0$ or $k =n$ a factor $Y_0$ is involved. This is since this is the only $Y_n$ that contains a singular factor $\overline{u}$ and we would like the result to be cubic. We can be more specific even: a term we are looking for only appears by the product of $\overline{u}$ from $Y_0$ and a term $\partial_x^{n-1}\overline{u}$ from $Y_n$, resulting in $u\overline{u}\partial_x^{n-1}\overline{u}$ for both $k = 0$ and $k = n$ in the sum.
\end{enumerate}
Accounting for the coefficients present and any edge-cases, we thus find that our coefficient function $K_n(k)$ fulfils the following recursion relation
\begin{equation*}
    K_{n+1}(k) = c\begin{cases}
    K_n(0) &\text{if $k=0$,}\\
    2K_n(0) + K_n(1) &\text{if $k = 1$,}\\
    K_n(k-1)+K_n(k) &\text{if $1 < k < n$,}\\
    2c^{n+1}+K_n(n-1) &\text{if $k = n$,}
    \end{cases}
\end{equation*}
for $n > 1$ and $0 \le k \le n-1$. One may easily verify, with the initial condition $K_1(0) = c^3$ being evident, this recursion relation is solved by
\begin{equation}
    K_{n}(k) = c^{n+2} \paren{2\binom{n}{k} - \delta_{0, k}},
\end{equation}
at least for $n > 1$. Note that the lack of symmetry here is no coincidence, as the terms whose coefficients are described by $K_n(k)$ shuffle derivatives between $u$ and $\overline{u}$ rather than two identical factors $u$.

Next we must investigate how the functional derivative $\frac{\delta}{\delta \overline{u}} \int_\R u Y_{n}\d{x}$ transforms these coefficients of terms in $Y_n$. For the readers convenience we recall the action of the functional derivative. If
\begin{equation*}
    F[\phi] = \int_\R f(\phi, \partial_x\phi, \partial_x^2 \phi, \ldots, \partial_x^N \phi) \d{x} \quad\text{one has}\quad \frac{\delta F}{\delta\phi} = \sum_{k=0}^N (-1)^{k}\partial_x^k \frac{\partial f}{\partial (\partial_x^k \phi)}.
\end{equation*}
So we must take care to account for the fact that every `bad' quartic term in the Hamiltonian $\int_\R u Y_n \d{x}$ is counted twice: once for the factor $\overline{u}$ without any derivatives lying upon it and possibly another time if the remaining $\overline{u}$ factor (that may carry derivatives). We will use the symbol $\mathcal{R}$ to account for terms that are not `bad' cubics and thus are not of importance for our analysis; it may differ from line to line. For $n > 1$ we figure
\begin{align}
    \frac{\delta}{\delta \overline{u}} \int_\R u Y_{n}\d{x} &= \sum_{k=0}^{n-1} (-1)^k \partial_x^k \frac{\partial (uY_n)}{\partial(\partial_x^k\overline{u})}\\
    &= \sum_{k=0}^{n-1} (-1)^k \partial_x^k (K_n(k) + \delta_{0,k}) |u|^2 (\partial_x^{n-1-k} u) + \mathcal{R}
    \intertext{Here we must be careful to account for the extra $1$ (which we do by introducing $\delta_{0, k}$), which appears when differentiating the term $u(\partial_x^{n-1}u)\overline{u}^2$ in the functional derivative. This nicely cancels with the Kronecker delta in the coefficient function $K_n(k)$. Next we use the classical Leibniz rule and interchange the order of summation:}
    &= \sum_{k=0}^{n-1}\sum_{\ell=0}^{k} (-1)^k 2c^{n+2}\binom{n}{k}\binom{k}{\ell} \overline{u} (\partial_x^{n-1-k+k-\ell}u)(\partial_x^{\ell} u) +\mathcal{R}\\
    &= 2c^{n+2} \sum_{\ell=0}^{n-1}\paren{\sum_{k=\ell}^{n-1}(-1)^k \binom{n}{k}\binom{k}{\ell}} \overline{u} (\partial_x^{n-1-\ell}u)(\partial_x^\ell u)+ \mathcal{R}\\
    &= 2c^{n+2}(-1)^{n+1} \sum_{\ell=0}^{n-1}\binom{n}{\ell} \overline{u} (\partial_x^{n-1-\ell}u)(\partial_x^\ell u) + \mathcal{R},
\end{align}
where in the final step we used a well-known summation identity for binomial coefficients.

This representation of $\frac{\delta}{\delta \overline{u}} \int_\R u Y_{n}\d{x}$ we may now use as the right-hand side in the definition of our evolution equations~\eqref{eq:abstract-dnls} in order to determine the coefficients we are interested in. Again we denote terms that are not of interest to us by use of the symbol $\mathcal{R}$, which may change from line to line:
\begin{align*}
    i\partial_t u &= 2\alpha_n\partial_x \frac{\delta}{\delta \overline{u}} \int_\R u Y_{n}\d{x} = \frac{4(-1)^{n+1}\alpha_n}{(2i)^{n+2}} \partial_x \sum_{\ell=0}^{n-1}\binom{n}{\ell} \overline{u} (\partial_x^{n-1-\ell}u)(\partial_x^\ell u) + \mathcal{R}\\
    &= \frac{4(-1)^{n+1}\alpha_n}{(2i)^{n+2}} \sum_{\ell=0}^{n-1}\binom{n}{\ell} \overline{u}\paren{(\partial_x^{n-\ell}u)(\partial_x^{\ell}u) + (\partial_x^{n-(\ell+1)u})(\partial_x^{\ell+1}u)} +\mathcal{R}\\
    &=\frac{4(-1)^{n+1}\alpha_n}{(2i)^{n+2}} \paren{\sum_{\ell=0}^{n-1} \binom{n}{\ell} \overline{u}(\partial_x^{n-\ell}u)(\partial_x^{\ell}u) + \sum_{\ell=1}^{n} \binom{n}{\ell-1} \overline{u}(\partial_x^{n-\ell}u)(\partial_x^{\ell}u)} +\mathcal{R}.
\end{align*}
The first and last terms of these sums respectively are both of the form $|u|^2\partial_x^n u$ so we may combine them. All other `bad' cubics appear in the sums twice by symmetry so we `fold over' the sum in order to see the actual coefficient. We omit the leading factor for space reasons.
\begin{align*}
    &(n+1)|u|^2(\partial_x^n u) + \sum_{\ell=1}^{n-1} \binom{n+1}{\ell} \overline{u}(\partial_x^{n-\ell}u)(\partial_x^{\ell}u) +\mathcal{R}\\
    &=(n+1)|u|^2(\partial_x^n u) + \sum_{\ell=1}^{\lfloor\frac{n-1}{2}\rfloor} \paren{\binom{n+1}{\ell} + \binom{n+1}{n-\ell}} \overline{u}(\partial_x^{n-\ell}u)(\partial_x^{\ell}u) +\mathcal{R}\\
\end{align*}
Using the identity $\binom{n+1}{\ell} + \binom{n+1}{n-\ell} = \binom{n+2}{\ell+1}$ we have now been able to completely determine the coefficients of the `bad' cubic terms appearing in the hierarchy equations. Noting that $\binom{n+2}{0+1} = \binom{n+2}{n+1} = n+2 = (n+1) - 1$ one may verify that the representation given in~\eqref{eq:dnls-coeff} is correct.
\end{proof}

Our next step in preparation of the proof of Proposition~\ref{prop:gauge-transformation} is figuring out which cubic nonlinear terms can be lifted by the gauge-transformation (and which coefficients lead to total cancellation of these terms). For this we will prove the following lemma in which it is established which `bad' cubic terms are generated by inserting a gauge-transformed function into the linear part of a dNLS hierarchy equation.
\begin{lemma}\label{lem:gauge-coeff}
    Let $j \in \N$ and $u$ be a solution of the $j$th dNLS hierarchy equation. We set $v := \mathcal{G}_-(u)$ to be its gauge-trans\-form. The coefficient of the `bad' cubic term $(\partial_x^{2j-1-\ell}u)\overline{u}(\partial_x^\ell u)$, in terms of $u$, appearing in $i\partial_t v + (-1)^{j+1}\partial_x^{2j} v$ is
    \begin{equation}
        i(-1)^{j+1} \paren{\binom{2j+1}{\ell+1} - \delta_{0, \ell} -\delta_{2j-1, \ell}}.
    \end{equation}
\end{lemma}
\begin{proof}
    We begin this proof by simple insertion of $v$ into the proposed linear part of a dNLS hierarchy equation and elementary calculation:
    \begin{align}\label{eq:gauge-inserted-eq}
        i\partial_t v + (-1)^{j+1}\partial_x^{2j} v &= G_u(i\partial_t u + (-1)^{j+1}\partial_x^{2j}u \\&+ u\int_{-\infty}^x u_t\overline{u} + u\overline{u_t}\d{\lambda} + i(-1)^j\sum_{k=0}^{2j-1} \partial_x^{2j-1-k}\paren{|u|^2\partial_x^k u} + \mathcal{R}).\notag
    \end{align}
    Here we have re-used the symbol $\mathcal{R}$ to denote higher-order terms and non-`bad' cubics and introduced the notation $G_u = \exp\paren{-i\int_{-\infty}^x |u(y)|^2 \d{y}}$ to simplify matters. Further cubic nonlinear terms may be produced by the integral, but only if the integrand is quadratic in $u$. Inserting the dNLS hierarchy equation that is solved by $u$ for the terms $u_t$ and $\overline{u_t}$ we see that the integrand is only quadratic for the linear dispersion term in the equation:
    \begin{align}
        u\int_{-\infty}^x u_t\overline{u} + u\overline{u_t}\d{\lambda} &= u\int_{-\infty}^x (i(-1)^{j+1}\partial_x^{2j}u)\overline{u} - u(i(-1)^{j+1}\partial_x^{2j}\overline{u}) \d{\lambda} +\mathcal{R}\\
        &= i(-1)^{j+1}u \int_{-\infty}^x (\partial_x^{2j}u)\overline{u} - u(\partial_x^{2j}\overline{u}) \d{\lambda}+\mathcal{R}.
    \end{align}
    The reader may now inductively verify the fact that this integral can be rewritten as
    \begin{equation*}
        i(-1)^{j+1}u \int_{-\infty}^x (\partial_x^{2j}u)\overline{u} - u(\partial_x^{2j}\overline{u}) \d{\lambda} = i(-1)^{j+1} u\sum_{k=0}^{2j-1}(-1)^k (\partial_x^{2j-1-k}u)(\partial_x^k \overline{u}).
    \end{equation*}
    Of the terms in this sum only the first one is a `bad' cubic term, so when we now return to~\eqref{eq:gauge-inserted-eq} the other terms of the sum may be absorbed into $\mathcal{R}$ and we are left with
    \begin{align}
        \eqref{eq:gauge-inserted-eq} = G_u(i\partial_t u &+ (-1)^{j+1}\partial_x^{2j}u + i(-1)^{j+1}|u|^2(\partial_x^{2j-1}u)\\ & + i(-1)^j\sum_{k=0}^{2j-1} \partial_x^{2j-1-k}\paren{|u|^2\partial_x^k u} + \mathcal{R})\notag\\
        = G_u(i\partial_t u &+ (-1)^{j+1}\partial_x^{2j}u + i(-1)^{j+1}|u|^2(\partial_x^{2j-1}u)\\ & + i(-1)^j\sum_{k=0}^{2j-1}\sum_{\ell=0}^{2j-1-k}\binom{2j-1-k}{\ell} \overline{u}(\partial_x^{2j-1-k-\ell} u)(\partial_x^{k+\ell} u) + \mathcal{R}).\notag
    \end{align}
    Now interchanging the sums and using a well-known identity $\sum_{\ell=k}^{2j-1}\binom{2j-1-k}{\ell-k} = \binom{2j}{\ell}$ for binomial coefficients we go on to write
    \begin{align}
        =G_u(i\partial_t u &+ (-1)^{j+1}\partial_x^{2j}u - i(-1)^{j}|u|^2(\partial_x^{2j-1}u)\\ & + i(-1)^j\sum_{\ell=0}^{2j-1} \frac{2j}{2j-\ell}\binom{2j-1}{\ell} \overline{u}(\partial_x^{2j-1-\ell} u)(\partial_x^{\ell} u) + \mathcal{R}).\notag
    \end{align}
    To account for the symmetry of the cubic terms with derivatives on the factors $u$ we again `fold-over' this sum so that we may read off the coefficients more comfortably:
    \begin{align}
        &=G_u(i\partial_t u + (-1)^{j+1}\partial_x^{2j}u - i(-1)^{j}|u|^2(\partial_x^{2j-1}u)\\ & + i(-1)^j\sum_{\ell=0}^{\lfloor \frac{2j-1}{2} \rfloor} \paren{\frac{2j}{2j-\ell}\binom{2j-1}{\ell} + \frac{2j}{\ell+1}\binom{2j-1}{2j-1-\ell}} \overline{u}(\partial_x^{2j-1-\ell} u)(\partial_x^{\ell} u) + \mathcal{R})\notag\\
        &=G_u(i\partial_t u + (-1)^{j+1}\partial_x^{2j}u \\&+ i(-1)^j\sum_{\ell=0}^{\lfloor \frac{2j-1}{2} \rfloor} \paren{\binom{2j+1}{\ell+1} - \delta_{0, \ell}} \overline{u}(\partial_x^{2j-1-\ell} u)(\partial_x^{\ell} u) + \mathcal{R})\notag
    \end{align}
    These coefficients coincide with the statement of this lemma so the proof is complete.
\end{proof}

Now all ingredients we need for the proof of Proposition~\ref{prop:gauge-transformation} are set in place.
\begin{proof}[Proof of Proposition~\ref{prop:gauge-transformation}]
    There isn't much left to argue: When applying the gauge-transformation $v = \mathcal{G}_-(u)$ and inserting $v$ into the linear part of a dNLS hierarchy equation, the way one recovers which equation $v$ solves is by using that $u$ solves a dNLS hierarchy equation and then rewriting all nonlinear terms in $v$ instead of $u$ by supplementing factors with the exponential function involved in the gauge-transformation and/or adding correctional higher-order terms.

    Since we have now found, between Proposition~\ref{prop:dnls-coeff} and Lemma~\ref{lem:gauge-coeff}, that the coefficients of the dNLS hierarchy equations and the gauge-transformation coincide (we remind the reader that for a dNLS hierarchy equation we set $n = 2j-1$ and our choice of $\alpha_{2j-1} = 2^{2j-1}$), we can be sure of the fact that at least before rewriting the nonlinear terms in terms of $v$, all `bad' cubic terms are cancelled by the gauge-transformation. In supplementing cubic terms with the exponential function form the gauge-transformation we do not suddenly turn them `bad' and higher-order terms that may need to be added (in order to account for cases where the derivative in a gauge-transformed nonlinear terms falls onto the exponential function) are of no concern to us.
\end{proof}

\subsection{Continuity of the gauge-transformation}

After having established that the use of the gauge-transformation absolves us of the most ill-behaved terms in dNLS hierarchy equations, we must also argue that it is compatible with our goal of well-posedness. More precisely we must exhibit its continuity, so that the gauge-transformation may be used to pull-back well-posedness results for gauged dNLS equations to well-posedness for dNLS hierarchy equations that we are actually interested in.

For well-posedness in Fourier-Lebesgue spaces continuity of the gauge-trans\-for\-ma\-tion had previously been established in the literature.

\begin{lemma}[\cite{AG2005}*{Lemma~3.3 and Remark~3.4}]\label{lem:gauge-continuity-hat}
    Let $s \ge \1{2}$ and $1 < r \le 2$. Then the gauge-transformation $\mathcal{G}_\pm : \hat H^s_r(\R) \to \hat H^s_r(\R)$ is Lipschitz continuous on bounded sets. The same holds true if the gauge-transformation is viewed as a map $\mathcal{G}_\pm : C(I, \hat H^s_r) \to C(I, \hat H^s_r)$ for an arbitrary interval $I \subset\R$.
\end{lemma}

Though even with well-posedness results for dNLS in modulation spaces already appearing in the literature, see~\cite{Guo2021}, where the gauge-transformation aided in simplifying the equation, the issue of its continuity does not seem to have been tackled. Thus we prove the following Lemma.

\begin{lemma}\label{lem:gauge-continuity-modulation}
    Let $2 \le p < \infty$ and $s > \1{2} - \frac{1}{p}$. Then the gauge-transformation $\mathcal{G}_\pm : M^s_{2, p}(\R) \to M^s_{2, p}(\R)$ is Lipschitz continuous on bounded sets. Moreover it is also continuous interpreted as a map $\mathcal{G}_\pm : C(I, M^s_{2, p}) \to C(I, M^s_{2, p})$ for an arbitrary interval $I \subset\R$.
\end{lemma}

\begin{remark}
    The regularity restriction $s > \1{2}-\1{p}$ is only natural since this is necessary for the embedding $M^s_{2,p} \subset L^2$ to hold, which in turn is necessary for the gauge-transformation to be well-defined.
\end{remark}

\begin{proof}[Proof of Lemma~\ref{lem:gauge-continuity-modulation}]
    In order to simplify notation we will only make the argument for $\mathcal{G}_+$, the minus-case works the same, and we also introduce the notation
    \begin{equation}
        G_u(x) = \exp\paren{i\int_{-\infty}^x |u(y)|^2 \d{y}} \quad\text{and}\quad \mathcal{I}(u)(x) = \int_{-\infty}^x |u(y)| \d{y}
    \end{equation}
    notwithstanding possible $t$ dependence of $u$, so the gauge-transformation may be written as $\mathcal{G}_+(u)(x) = G_u u(x)$.

    We will be following an argument given in~\cite{Herr2006}*{Appendix~A}, thus we will establish an estimate
    \begin{equation}\label{eq:general-gt-est}
        \norm{(G_v - G_w)u}_{M^s_{2,p}} \lesssim e^{c \norm{v}^2_{M^s_{2,p}} + c \norm{w}^2_{M^s_{2,p}}} \norm{v+w}_{M^s_{2,p}} \norm{v-w}_{M^s_{2,p}} \norm{u}_{M^s_{2,p}}.
    \end{equation}
    With~\eqref{eq:general-gt-est} we may argue the Lipschitz continuity of $\mathcal{G}_+$ for functions $u, w \in B_r(0) \subset M^s_{2,p}$ as follows
    \begin{align*}
        \norm{\mathcal{G}_+(u) - \mathcal{G}_+(v)}_{M^s_{2,p}} &\lesssim \norm{(G_u - G_v)u}_{M^s_{2,p}} + \norm{(G_v - 1)(u-v)}_{M^s_{2,p}} + \norm{u-v}_{M^s_{2,p}}\\
        &\lesssim (re^{2cr^2} + re^{cr^2} + 1)\norm{u - v}_{M^s_{2,p}} \lesssim_r \norm{u - v}_{M^s_{2,p}}.
    \end{align*}

    We are left to argue~\eqref{eq:general-gt-est}. First we use the generalised Leibniz rule for modulation spaces~\eqref{eq:modulation-leibniz} which results in
    \begin{equation}\label{eq:gt-est-leibniz}
        \norm{(G_v - G_w)u}_{M^s_{2,p}} \lesssim \norm{G_v - G_w}_{M^s_{\infty,\tilde{p}}} \norm{u}_{M_{2,2}} + \norm{G_v - G_w}_{M_{\infty,1}} \norm{u}_{M^s_{2,p}},
    \end{equation}
    where $\1{p'} = \1{2} + \1{\tilde p'}$. Looking at the second term in the sum we must estimate $G_v - G_w$ in the $M_{\infty, 1}$ norm. We use the algebra property of this space and the power series expansion of the exponential function to arrive at
    \begin{align*}
        &\norm{G_v - G_w}_{M_{\infty,1}}\\ &\lesssim \norm{\mathcal{I}(|v|^2 - |w^2|)}_{M_{\infty, 1}} \sum_{k=1}^\infty \1{k!}\sum_{j=0}^{k-1} (c\norm{\mathcal{I}(|v|^2)}_{M_{\infty,1}})^j (c\norm{\mathcal{I}(|w|^2)}_{M_{\infty,1}})^{k-j-1}\\
        &\lesssim \norm{\mathcal{I}(|v+w||v-w|)}_{M_{\infty, 1}} \exp(c\norm{\mathcal{I}(|v|^2)}_{M_{\infty, 1}} + c\norm{\mathcal{I}(|w|^2)}_{M_{\infty, 1}}).
    \end{align*}
    From here, if we are now able to argue the bilinear estimate
    \begin{equation}
        \norm{\mathcal{I}(f g)}_{M_{\infty, 1}} \lesssim \norm{f}_{M^s_{2,p}} \norm{g}_{M^s_{2, p}},
    \end{equation}
    we arrive at our desired~\eqref{eq:general-gt-est}.
    We look at two cases depending on the magnitude of the frequency of $\mathcal{I}(f g)$ because of the singularity introduced by $\mathcal{I}$ at low frequencies.
    \begin{enumerate}[wide]
        \item \textbf{low frequencies}: Since here we only have finitely many terms in the outer $\ell^1$ norm we may estimate by $L^\infty$, use Hölder's inequality and a Sobolev-type embedding for modulation spaces~\eqref{eq:modulation-sobolev} to arrive at
        \begin{equation}
            \norm{P_{1} \mathcal{I}(fg)}_{M_{\infty, 1}} \lesssim \norm{\mathcal{I}(fg)}_{L^\infty} \lesssim \norm{f}_{L^2} \norm{g}_{L^2} \lesssim \norm{f}_{M^s_{2, p}} \norm{g}_{M^s_{2, p}},
        \end{equation}
        since $s > \1{2}-\1{p}$.
        \item \textbf{high frequencies}: In this situation we may replace $\mathcal{I}(fg)$ with a Bessel potential operator
        \begin{align*}
            \norm{P_{> 1} \mathcal{I}(fg)}_{M_{\infty, 1}} &\lesssim \norm{J^{-1} (fg)}_{M_{\infty, 1}} \lesssim \norm{fg}_{M_{\infty, r}}\\ &\lesssim \norm{f}_{M_{\infty,\rho}} \norm{g}_{M_{\infty,\rho}} \lesssim \norm{f}_{M^s_{2,p}} \norm{g}_{M^s_{2,p}}
        \end{align*}
        where we then use Hölder's inequality with $r = \infty-$ in the outer $\ell^1$ norm and then Hölder's inequality again in the outer norm, with $\1{r'} = \frac{2}{\rho'} \Leftrightarrow \1{\rho} = \1{2}+$. Finally we use a Sobolev-type embedding for modulation spaces~\eqref{eq:modulation-sobolev} which requires $s > \1{\rho} - \1{p} = \1{2} - \1{p} +$.
    \end{enumerate}

    Now we turn to the first term in the sum in~\eqref{eq:gt-est-leibniz}. The $M_{2, 2} = L^2$ norm of $u$ may again be estimated by $\norm{u}_{M^s_{2, p}}$ due to the Sobolev-type embedding~\eqref{eq:modulation-sobolev}. For the other factor we argue similarly to the above, noting that $M^s_{\infty, \tilde p}$ is also an algebra since $s > \1{2} - \1{q} = \1{\tilde p'}$, though this time we require a bilinear estimate of the form
    \begin{equation}
        \norm{\mathcal{I}(f g)}_{M^s_{\infty, \tilde p}} \lesssim \norm{f}_{M^s_{2,p}} \norm{g}_{M^s_{2, p}}.
    \end{equation}
    For low frequencies we may reuse our argument from above, since in that case $\norm{P_1 \mathcal{I}(f g)}_{M^s_{\infty, \tilde p}} \lesssim \norm{\mathcal{I}(f g)}_{M_{\infty, 1}}$, whereas for high frequencies we argue
    \begin{equation}
        \norm{P_{> 1} \mathcal{I}(f g)}_{M^s_{\infty, \tilde p}} \lesssim \norm{f g}_{M^{s-1}_{\infty, \tilde p}} \lesssim \norm{f g}_{H^{s-1+s'}}
    \end{equation}
    where $s' > \1{\tilde p} - \1{2} = \1{p}$. Then $s -1 + s' = -\1{2}+$ and we may use a Sobolev embedding and Hölder's inequality
    \begin{equation}
        \lesssim \norm{fg}_{L^{1+}} \lesssim \norm{f}_{L^{2+}} \norm{g}_{L^{2+}} \lesssim \norm{f}_{M^s_{2,p}} \norm{g}_{M^s_{2, p}},
    \end{equation}
    where in the final inequality we used a Sobolev-type embedding for modulation spaces~\eqref{eq:modulation-sobolev} again.

    The claim of continuity of $\mathcal{G}_\pm$ on $C(I, M^s_{2,p})$ follows by replacing the $M^s_{2,p}$ norms by $L^\infty_t M^s_{2,p}$ norms. We omit the details.
\end{proof}

\section{Statement of results}\label{sec:statement-of-results}

\subsection{Prior work}
Before we state our main results let us quickly review the literature regarding low regularity well-posedness results for the dNLS equation itself as well as the fourth order dNLS hierarchy equation ($j=2$). To the author's knowledge the other, higher-order, equations part of the dNLS hierarchy do not yet appear in the literature. Giving a complete account of the well-posedness theory (especially concerning results of the inverse scattering community) though is beyond our scope, so we will focus mostly on comparable results to our own.

As is unsurprising the dNLS equation (and variants of it) were first tackled using the energy method, see~\cites{MR621533,MR634894}, achieving local well-posedness for initial data in $H^s$ (independent of the underlying geometry) for $s > \frac{3}{2}$.

On the line these results were later improved in~\cite{HayashiOzawa1992} to cover both local and global well-posedness (thanks to energy conservation) in $H^1(\R)$, under the restriction that the mass of the initial data be smaller than $2\pi$. Already here the gauge-transformation was used in order to make the equation approachable using dispersive PDE techniques.

In parallel it was begun to utilise the dispersive character of the equation\footnote{Not just the dNLS equation was considered here, but a rather large class with arbitrary polynomial nonlinearity involving derivatives.} in~\cite{MR1230709}, where a variant of Kato smoothing together with a maximal function estimate was used in order to establish small data local well-posedness in $H^{\frac{7}{2}}(\R)$. 

Using multilinear refinements of smoothing estimates for the Schrödinger propagator together with $X_{s,b}$ spaces the local well-posedness result could be pushed down to $H^\1{2}(\R)$. See~\cite{Takaoka1999}. In a subsequent paper~\cite{Takaoka2001}, for $s > \frac{32}{33}$, these newly constructed local solutions were extended globally using the splitting-argument, which was initially developed by Bourgain. It was also recorded that, since the flow fails to be thrice continuously differentiable for $s < \1{2}$, there was no hope in further improving the local result on the line using the contraction mapping theorem alone.

More dire still, after in~\cite{BiagioniLinares2001} it had been established using exact soliton solutions to the dNLS equation, that the flow of the dNLS equation cannot be uniformly continuous for $s < \1{2}$.

On the front of improvements to global well-posedness, after a refined version of the splitting-argument, today usually referred to as the I-method, had been developed, the global result could be pushed down to almost match the (now known to be optimal, using fixed-point methods) local result. That is, in~\cites{MR1871414,MR1950826} it was proven that the dNLS equation is globally well-posed in $H^s(\R)$ for $s > \1{2}$, conditioned on a mass below $2\pi$.

Global well-posedness in the endpoint $s = \1{2}$, under the same mass restriction as previously, was later shown by different authors~\cite{MR2823664}, again using the I-method, but additionally using a resonant decomposition technique to better control a singularity arising from resonant interactions.

Trying to push the local result further towards the scaling critical space, Fourier-Lebesgue spaces were employed in~\cite{AG2005}, where then local well-posedness was achieved in $\hat{H}^s_r(\R)$ for $s \ge \1{2}$ and $2 \ge r > 1$. This covers the entire scaling sub-critical configuration of parameters.

As modulation spaces moved into focus of the dispersive PDE community these spaces were also employed in order to move well-posedness results closer to the scaling critical space. In~\cite{Guo2021} local well-posedness for initial data in $M^\1{2}_{2,q}$ for $4 \le q < \infty$ was proven. Here $M^\1{2}_{2,\infty}$ is understood to be the critical space, even though modulation spaces are not well-behaved under transformations of scale. It is of note, that in the previously cited work the continuity of the gauge-transform in appropriate modulation spaces was not discussed. We resolve this issue with Lemma~\ref{lem:gauge-continuity-modulation}.

The mass restriction though, that had so far been part of all global results, turned out to be a mere technically arising restriction. This, over the course of~\cites{MR3198590,MR3393674}, could be lifted from $2\pi$ to $4\pi$ using the sharp version of the Gagliardo-Nirenberg inequality for global solutions in $H^1(\R)$. This result was later then extended to also cover the full range of possible local results, i.e. in~\cite{MR3583477} it was shown that, under the lighter mass restriction of $4\pi$ solutions with initial data in $H^\1{2}(\R)$ extended globally.

The most recent and extensive results concerning the low-regularity well-posedness theory of the dNLS equation were achieved with methods stemming from the equation's complete integrability. Using those techniques it was possible to prove global well-posedness held in the scaling critical space $L^2(\R)$ with no restriction on the mass of the initial data~\cites{MR4565673,MR4628747}. Moreover, those two papers and references therein give a nice, general overview of recent well-posedness results for the dNLS equation achieved with inverse scattering/complete integrability.

Since we are less concerned with results pertaining to periodic initial data we will stick to headlines only. It was only with~\cite{Herr2006} that a version of the gauge-transformation was found, such that the dNLS equation could be attacked using fixed-point methods on the torus. Here the optimal local well-posedness result could immediately be paralleled (despite the lack of strong smoothing effects of the Schrödinger group), i.e. well-posedness for initial data in $H^\1{2}(\T)$ was achieved. The argument used the $L^4$ Strichartz estimate extensively. Ill-posedness, in the sense of failure of thrice differentiability of the flow below $s = \1{2}$ is contained in the same work.

Here again, Fourier-Lebesgue spaces (that in the periodic setting coincide with modulation spaces) could be used in order to push the local well-posedness result nearer the scaling critical space. Over the course of~\cites{MR2390318,AG2005} well-posedness could be extended to initial data in $\hat{H}^s_r(\T)$ for $s =\1{2}$ and $2 \ge r > \frac{4}{3}$. Covering local well-posedness in the remainder of the subcritical range, that is $\hat{H}^\1{2}_r(\T)$ for $r > 1$, was then achieved in~\cite{MR4259382}.

Much fewer works have so far dealt with any higher-order dNLS hierarchy equations. We mention~\cite{MR2818712}, where a well-posedness results covering the fourth-order dNLS equation is proven. Specifically small data local well-posedness for data in $H^s(\R)$, $s > 4$, is established.

This was later improved in~\cite{Ikeda2021} to small data well-posedness for data in $H^1(\R)$. The dNLS hierarchy equation is also explicitly mentioned in this work. Further low-regularity well-posedness results covering the higher-order dNLS hierarchy equations are not present in the literature, to the author's best knowledge.

\subsection{Main results}
First we consider a general Cauchy problem for an evolution equation of the form
\begin{equation}\label{eq:general-cauchy}
    \begin{cases}
    i\partial_t u + (-1)^{j+1} \partial_x^{2j} u = F(u)\\
    u(t=0) = u_0
    \end{cases},
\end{equation}
where we are able to derive the following well-posedness theorems for data in Fourier-Lebesgue and modulation spaces regarding the dNLS hierarchy.

\begin{theorem}\label{thm:hierarchy-wellposed-hat-spaces}
    Let $j \ge 2$ and~\eqref{eq:general-cauchy} be the $j$th dNLS hierarchy equation.
    If $1 < r \le 2$ and $s \ge \1{2} + \frac{j-1}{r'}$, the Cauchy problem~\eqref{eq:general-cauchy} with initial data $u_0 \in \hat H^s_r(\R)$ is locally well-posed, with the solution map being Lipschitz continuous on bounded sets.
\end{theorem}

For $j = 1$ this theorem corresponds to the well-posedness of the dNLS equation in Fourier-Lebesgue spaces on the line and is already known in the literature~\cite{AG2005}.

\begin{remark}
    The condition $r \le 2$ appears naturally in this context, because of the use of the gauge-transformation, the well-definedness of which requires $L^2 \supset \hat H^s_r$.
\end{remark}

\begin{theorem}\label{thm:hierarchy-wellposed-modulation-spaces}
    Let $j \ge 2$ and~\eqref{eq:general-cauchy} be the $j$th dNLS hierarchy equation. Then for $2 \le p < \infty$ and $s \ge \frac{j}{2}$, the Cauchy problem~\eqref{eq:general-cauchy} with initial data $u_0 \in M^s_{2,p}(\R)$ is locally well-posed with a solution map that is Lipschitz continuous on bounded subsets.
\end{theorem}

Again, for $j=1$ (and $4 \le p$) the well-posedness of the dNLS equation in modulation spaces on the line can already be found in the literature, see~\cite{Guo2021}, though there the continuity of the gauge-transformation is not discussed.

\begin{remark}\label{remark:uniqueness}
It was already noted immediately after stating Theorem~3 in~\cite{AG2005}, that the uniqueness statement in the preceding well-posedness theorems (and in~\cite{AG2005}) was to be carefully interpreted. Due to the gauge-transformation, uniqueness of a solution $u$ only holds with respect to other solutions $v$ that fulfil the artificial seeming condition that $\mathcal{G}_-v$ must also solve the associated gauge-transformed equation (corresponding to a dNLS hierarchy equation).
\end{remark}

Most noticeable about these theorems, in comparison with their analogues for the NLS hierarchy equations~\cite{Adams2024}*{Theorems~3.1 and~3.2}, is the lower regularity of the solution map: being merely Lipschitz continuous rather than analytic. This is due to the fact that Theorems~\ref{thm:hierarchy-wellposed-hat-spaces} and~\ref{thm:hierarchy-wellposed-modulation-spaces} are derived from the following well-posedness theorems concerned with gauged dNLS equations. We remind the reader that a gauged dNLS equation contains no `bad' cubic nonlinear terms, i.e. where no derivatives fall on the complex conjugated term.

Since we only know the gauge-trans\-for\-ma\-tion to be a bi-Lipschitz continuous map on bounded sets, see Lemmas~\ref{lem:gauge-continuity-hat} and~\ref{lem:gauge-continuity-modulation}. Hence the pull-back of the solution map is not analytic but merely Lipschitz continuous.

\begin{theorem}\label{thm:gauged-wellposed-hat}
    Let $j \ge 2$ and~\eqref{eq:general-cauchy} be a gauged dNLS equation. Then
    \begin{enumerate}
        \item if $1 < r \le 2$ and $s \ge \1{2} + \frac{j-1}{r'}$, the Cauchy problem~\eqref{eq:general-cauchy} with initial data $u_0 \in \hat H^s_r(\R)$ is locally well-posed with an analytic solution map,
        \item if additionally~\eqref{eq:general-cauchy} contains no cubic nonlinear terms, $1 < r \le 2$ and $s > \1{r} - \1{2}$, the Cauchy problem with initial data $u_0 \in \hat H^s_r(\R)$ is locally well-posed with an analytic solution map.
    \end{enumerate}
\end{theorem}

\begin{theorem}\label{thm:gauged-wellposed-modulation}
    Let $j \ge 2$ and~\eqref{eq:general-cauchy} be a gauged dNLS equation. Then
    \begin{enumerate}
        \item if $2 \le p < \infty$ and $s \ge \frac{j}{2}$, the Cauchy problem~\eqref{eq:general-cauchy} with initial data $u_0 \in M^s_{2,p}(\R)$ is locally well-posed with an analytic solution map,
        \item if additionally~\eqref{eq:general-cauchy} contains no cubic nonlinear terms and $2 \le p \le \infty$, let $k \ge 2$ be the smallest index for which $c_{k,\alpha} \not= 0$ (in~\eqref{eq:gauged-dnls-eqn}) for a choice of $\alpha \in \N_0^{2k+1}$. Then for $s > \1{2} + \1{4k} - \frac{2k+1}{2kp}$, the Cauchy problem with initial data $u_0 \in M^s_{2,p}(\R)$ is locally well-posed with an analytic solution map.
    \end{enumerate}
\end{theorem}

\begin{remark}
    Theorem~\ref{thm:gauged-wellposed-modulation} has further extensions: besides the (also called) gauge-invariant (with respect to multiplication with a constant phase-factor $u \mapsto e^{i\theta}u$) distribution of complex conjugates in the nonlinear terms others are possible. Only for the cubic term $|u|^2u$ is the necessary distribution of complex conjugates with our arguments, ignoring derivatives. For the higher-order terms an arbitrary distribution of complex conjugates is possible.

    A similar, if weaker, statement regarding the arbitrariness of distribution of complex conjugates in the nonlinear terms is true of Theorem~\ref{thm:gauged-wellposed-hat}. For example, the proof of Proposition~\ref{prop:quintic-estimate-hat-endpoint} shows that the statement of the theorem still holds true, if only as few as two factors are complex conjugates in a quintic or higher-order nonlinear term.

    Though we do not pursue an in-depth showcasing of which distributions of complex conjugates are covered by our arguments, as the gauge-invariant (see above) nonlinearities are most canonical.
\end{remark}

We derive these theorems by means of proving multilinear estimates in $\hat X^r_{s,b}$ and $X^p_{x, b}$ spaces for the nonlinear terms in the equations. Definitions of their respective norms are given by $\norm{f}_{\hat X^r_{s,b}} = \norm{\JBX[\tau - \xi^{2j}]^b\JBX[\xi]^s\F_{x,t}f}_{L^{r'}_{xt}}$ and $\norm{f}_{X^p_{s,b}} = \norm{\JBX[n]^s \norm{\Box_n f}_{\hat X^2_{0, b}}}_{\ell^p_n(\Z)}$, where $(\Box_n)_{n\in\Z}$ is a fixed choice of uniform frequency decomposition operators. Properties of these function spaces were covered in~\cite{Adams2024}*{Section~1.2}. Combined with the contraction mapping principle such estimates lead to local well-posedness in Fourier-Lebesgue and modulation spaces respectively. Using such estimates to obtain local well-posedness results is a well-known technique initially investigated in~\cites{Bourgain1993-1,Bourgain1993-2}. We omit specific details of the connection between non- or multilinear estimates and well-posedness, but direct the uninitiated reader to~\cites{AGDiss,Grünrock2004} for an overview and necessary adaptations in order to deal with Fourier-Lebesgue and modulation spaces (rather than just Sobolev spaces).

In contrast with our well-posedness results given in the preceding theorems, we are also able to derive a number of ill-posedness results regarding the hierarchy equations in conjunction with the techniques that we are utilising. In particular the following two theorems show that no direct application of the contraction mapping theorem can lead to well-posedness for non-periodic initial data below the regularities at which we establish local well-posedness, i.e. our well-posedness results are optimal in this sense.

\begin{theorem}\label{thm:illposed-line-C3-hat}
    For any $j \ge 2$, $1 \le r \le \infty$ and $s < \1{2} + \frac{j-1}{r'}$ the flow map $S : \hat H^s_r(\R) \times (-T, T) \to \hat H^s_r(\R)$ of the Cauchy problem for the $j$th dNLS hierarchy equation cannot be thrice continuously differentiable.
\end{theorem}

\begin{theorem}\label{thm:illposed-line-C3-modulation}
    For any $j \ge 2$, $1 \le p,q \le \infty$ and $s < \frac{j}{2}$ the flow map $S : M^s_{p,q}(\R) \times (-T, T) \to M^s_{p,q}(\R)$ of the Cauchy problem for the $j$th dNLS hierarchy equation cannot be thrice continuously differentiable.
\end{theorem}

\begin{remark}\label{rem:C3-illposedness-line}
    The preceding two theorems are phrased for the dNLS hierarchy equations themselves. This turns out to be an unnecessary restriction though. As the proofs will show we are only concerned with cubic nonlinear terms and with that we may also ignore the distribution of derivatives within them. The latter stems from the fact that the ill-posedness result is derived from a high-high-high interaction between the three factors, so that derivatives may be shifted arbitrarily between factors anyway.

    Thus there is still a lot of leeway in phrasing the ill-posedness theorems for more general classes of equations. Since we have not defined a name for this explicit class we refrain from complicating the theorem by trying to be as general as possible in its phrasing. Suffice it to say that our ill-posedness theorems still hold, so long as a cubic nonlinear term (in an equation paralleling~\eqref{eq:general-cauchy}) with $2j-1$ derivatives placed upon it is present in the equation. In particular this also includes the class of gauged dNLS equations.
\end{remark}

Moving from the realm of non-periodic initial data to the periodic problem, we can be sure that no (direct) application of the contraction mapping theorem will lead to any positive results concerning the fourth-order hierarchy equation. Of course, this suggests that a similar result also holds for all higher-order equations. This would mean that merely the dNLS equation itself can be attacked using fixed-point techniques with periodic initial data.
\begin{theorem}\label{thm:illposed-torus-C3}
    For any $1 \le r \le \infty$ and $s \in \R$ the flow map $S : \hat H^s_r(\T) \times (-T, T) \to \hat H^s_r(\T)$ of the Cauchy problem for the fourth-order (i.e. $j=2$) dNLS hierarchy equation cannot be thrice continuously differentiable.
\end{theorem}
Weakening the regularity requirements for the initial data we are able to showcase that the situation regarding the regularity of the flow map is even worse. This then also generalises to an arbitrary higher-order hierarchy equation, strengthening our conviction that low-regularity well-posedness on the torus is out of reach for any of the hierarchy equations, except for dNLS itself.
\begin{theorem}\label{thm:illposed-torus-uniform}
    For any $j \ge 2$, $1 \le r \le \infty$ and $s < j - 1$ there exists a gauged dNLS equation (i.e. choice of coefficients $c_{k, \alpha}$) such that for the Cauchy problem~\eqref{eq:general-cauchy} the flow map $S : \hat H^s_r(\T) \times (-T, T) \to \hat H^s_r(\T)$ cannot be uniformly continuous on bounded subsets.
\end{theorem}
The corresponding proof for our ill-posedness theorems on the torus consist of the Propositions~\ref{prop:illposed-torus-C3} and~\ref{prop:illposed-torus-uniform} respectively.

We point out that these ill-posedness results, seemingly only regarding Cauchy problems in Fourier-Lebesgue spaces, suffice to also rule out well-posedness in modulation spaces on the torus. As in this periodic geometry the two families of function spaces coincide.

\subsection{Global well-posedness for the dNLS hierarchy}

Unfortunately, in contrast with the situation for the NLS hierarchy equations, we do not have a family of conservation laws equivalent to $H^s$ norms for every $s > -\1{2}$, as were constructed in~\cite{Koch2018}, at our disposal. Hence, for the moment, we are only able to upgrade our local solution to dNLS hierarchy equations to global ones at integer regularity levels. This leads to a discrepancy of at most half a derivative (exactly for the odd-indexed dNLS hierarchy equations) between our local result and the corresponding global continuation result of the solution.

\begin{theorem}\label{thm:global-wellposed}
    Let $j \ge 2$. If the initial data $u_0 \in H^{\lceil \frac{j}{2} \rceil}(\R)$ has sufficiently small $L^2$ norm, the solution of the $j$th dNLS hierarchy equation, constructed in Theorem~\ref{thm:hierarchy-wellposed-hat-spaces}, extends globally in time. In other words, one has small mass global well-posedness of the $j$th hierarchy equation with initial data in $H^{\lceil \frac{j}{2} \rceil}(\R)$. The bound on the mass depends on $j$, but not on the size of the $H^{\lceil \frac{j}{2}\rceil}$-norm of the initial data.
\end{theorem}
\begin{proof}
    We extend the previously constructed local solutions classically by utilising a-priori estimates that we derive from Hamiltonians $I_n$ as in~\eqref{eq:hamiltonian-recusion}. We remind the reader, that since we are dealing with a completely integrable hierarchy, the Hamiltonians of the hierarchy equations pairwise (Poisson) commute and are thus conserved along the flow of each other.
    The statement of this theorem holds true if we manage to derive an a-priori estimate on the $H^k$ norm of a solution (of an arbitrary dNLS hierarchy equation), for $k \in \N$.

    Guided by Lemma~\ref{lem:Yn-properties}, in order to derive an a-priori estimate on the level of $H^k$ we take a closer look at $I_{2k} = \int_\R uY_{2k}\d{x}$, where the `leading term' (up to constants) is given by $u \partial_x^{2k} \overline{u}$. By partial integration this term becomes equivalent to the homogeneous $\dot{H^k}$ norm. That the $L^2$ norm is conserved along the dNLS hierarchy equations' flows is well known. So what remains, until we may assert our desired a-priori bound on the $H^k$ norm, is to argue that the other terms in the Hamiltonian $I_{2k}$ cannot interfere with/cancel the leading term $|\partial_x^k u|^2$. That is, so far we have argued
    \begin{equation*}
        |I_{2k}| \gtrsim \norm{u}_{\dot{H}^k}^2 - |\text{higher order terms}|
    \end{equation*}
    and still need to ensure that the higher order terms can be controlled by a fraction (less than 1) of $\norm{u}_{\dot{H}^k}^2$.

    Take such a higher-order term of the Hamiltonian $I_{2k}$, which in general will be of the form $\prod_{i=0}^{m}(\partial_x^{\alpha_{2i}}u)(\partial_x^{\alpha_{2i+1}}\overline{u})$, for $\alpha \in \N_0^{2m+2}$ with $|\alpha| = 2k-m$ and $1 \le m \le 2k$. (In fact, from Lemma~\ref{lem:Yn-properties}, we know more about the structure: one of the factors $u$ will always be without a derivative placed upon it. But we ignore this additional bit of information at this point.) Since there are strictly less than $2k$ total derivatives, there will be at most a single factor that has more than $k$ derivatives placed upon it. Again, with partial integration, we may adjust such terms of the Hamiltonian so that every term has at most $k$ derivatives lying upon it, with at most a single one with exactly $k$ derivatives.

    Now we may apply Hölder's inequality to such higher-order terms in the Hamiltonian (with at most $k$ derivatives on any term) ensuring that, if there exists a factor $u$ with $k$ derivatives placed upon it, we put it in $L^2$. For the remaining factors with strictly fewer than $k$ derivatives it doesn't matter which $L^p$ they land in, so long as $p \ge 2$ (which is always possible, since we have $2m+2 \ge 4$ factors).

    We are now prepared to apply a special case of the Gagliardo-Nirenberg inequality~\eqref{eq:gagliardo-nirenberg}. In particular we will be choosing $p=q=2$, $\ell$ corresponds to the order of derivatives $\alpha_i$ placed on our factors and in our situation $n=1$ holds. This leads us to deriving $\theta = \frac{1}{2k} + \frac{\alpha_i}{k} - \frac{1}{rk}$. The inequality then reads
    \begin{equation*}
        \norm{\partial_x^{\alpha_i} u}_{L^r} \lesssim \norm{u}_{\dot{H}^{k}}^\theta \norm{u}_{L^2}^{1-\theta}.
    \end{equation*}
    Applying this inequality to every factor in a higher-order term $\prod_{i=0}^{2m+1} \norm{\partial_x^{\alpha_{i}}u}_{L^{p_i}}$ we are interested in the resulting exponent for $\norm{u}_{\dot{H}^{k}}$. We may calculate this as follows
    \begin{equation*}
        \sum_{i=0}^{2m+1} \frac{1}{2k} + \frac{\alpha_i}{k} - \frac{1}{p_i k} = \frac{2m+2}{2k} + \frac{|\alpha|}{k} - \1{k} = \frac{m}{k} + \frac{2k-m}{k} = 2,
    \end{equation*}
    where we have used the fact $\sum_{i=0}^{2m+1} \1{p_i} = 1$ and $|\alpha| = 2k-m$. For reasons of homogeneity we know the exponent of $\norm{u}_{L^2}$ must thus be $2m$.

    Hence at this point we have argued for an a-priori estimate of the form
    \begin{equation}
        I_{2k} \gtrsim \norm{u}_{H^k}^2(1 - c\norm{u}_{L^2}^{2m}) = \norm{u}_{H^k}^2(1 - c\norm{u_0}_{L^2}^{2m}) \gtrsim \norm{u}_{H^k}^2,
    \end{equation}
    where $c$ is a fixed constant, depending on the coefficients in the dNLS hierarchy equation (corresponding to the choice of $j \in \N$). The final inequality holds for a sufficiently small bound on the $L^2$ norm of the initial data. We have thus successfully argued for an a-priori estimate on the $H^k$ norm of solutions of dNLS hierarchy equations, conditioned on a sufficiently small initial mass.
\end{proof}

\subsection{Discussion}

Before moving on to proving our well- and ill-posedness results given in the previous subsection, we would like to discuss their merits and how they fit into the existing literature.

Let us begin by mentioning that our results show, that we have achieved optimal local well-posedness within the realm of our techniques, excluding the respective scaling critical Fourier-Lebesgue and modulation spaces. Specifically Theorems~\ref{thm:illposed-line-C3-hat} and~\ref{thm:illposed-line-C3-modulation} rule out the possibility of using fixed-point methods to improve upon the well-posedness theory of the dNLS hierarchy equations beyond what we have achieved. This does not preclude the possibility of using, say, the complete integrability of those equations to lower the regularity threshold on initial data while still achieving local well-posedness. As was already implemented for the dNLS equation itself~\cites{MR4565673,MR4628747} and recently the KdV hierarchy equations~\cite{KochKlaus2023}. Though this approach comes with the usual caveat that the flow will be rather irregular, i.e. merely continuous, rather than Lipschitz as in our Theorems~\ref{thm:hierarchy-wellposed-hat-spaces} and~\ref{thm:hierarchy-wellposed-modulation-spaces}.

On the front of global well-posedness we were able to exploit the Hamiltonians that are conserved along the flow of dNLS hierarchy equations in order to extend our local solutions globally, at least for initial data in Sobolev spaces at integer regularities. This leaves a gap of at most half a derivative between our local and global results. It seems likely that with an application of the (first-generation) I-method it would be possible to close this gap. More generally, extending solutions globally off the scale of Sobolev spaces (i.e. Fourier-Lebesgue or modulation spaces)  presents an interesting problem for further research.

We mention at this point that our local theory extends the previous best result concerning the fourth-order dNLS hierarchy equation from~\cite{Ikeda2021}, lifting the necessity of small data. With Theorem~\ref{thm:global-wellposed} we extend these local solutions globally. One point of interest is, that the authors of~\cite{Ikeda2021} manage to achieve their result without the use of any gauge-transformation. This seems to stem from their ability to exploit the special position of one of the derivatives in the nonlinearity $\partial_x$ being in front of every product term. See also~\eqref{eq:dnls-hierarchy} where we have also mentioned this fact. Further research into this possible exploitation may lead to subsequent further improvement of the regularity of the flow (of dNLS hierarchy equations), if one can do without the gauge-transformation.

The worsening of the lower bound for well-posedness by half a derivative in Sobolev spaces with every step up in the dNLS hierarchy ($j \rightarrow j+1$) is consistent with what can be observed for the similar situations of the NLS~\cite{Adams2024} or mKdV hierarchy~\cite{AGTowers}.

As is unsurprising, considering the ill-posedness results already for the NLS hierarchy on the torus~\cite{Adams2024}, the situation for low-regularity well-posedness theory of dNLS hierarchy equations on the torus is dire. One must hope that renormalisation/Wick-ordering or methods of complete integrability can be used in order to achieve any kind of result in this setting.

Regarding ill-posedness for the nonperiodic setting it has turned out to be surprisingly more difficult to achieve a general $C^0_\text{unif}$ ill-posedness result for the dNLS hierarchy compared with either the NLS or mKdV hierarchy. Explicit soliton solutions for dNLS, which were used in~\cite{BiagioniLinares2001} to show the failure of uniform continuity of the flow, are already very delicately constructed functions (evident from the complex choice of coefficients involved). Searching the literature for soliton solutions of higher-order dNLS hierarchy equations yielded only~\cite{MR4348344}, which due to their evidently even more complex structure and little resemblance to the solitons of dNLS suggest that this is a difficult problem to solve in full generality.

We end this subsection by mentioning that, to the author's best knowledge, we are also the first to achieve insight into the structure of coefficients in nonlinear terms in hierarchy equations stemming from completely integrable systems, beyond knowledge of a finite number of hierarchy equations. In particular referring to Proposition~\ref{prop:dnls-coeff}, where we derived a closed form expression for the coefficients of certain nonlinear terms appearing in the dNLS hierarchy equations. Extending such results to the rest of the nonlinear terms, or more generally other hierarchies, is of great interest. This would enable more delicate analysis regarding if the complete integrability structure of the equations has significant influence on the optimal well-posedness results that can be achieved with fixed point methods.


\section{Known Estimates}\label{sec:known-estimates}

In order to derive our well-posedness theorems, see Section~\ref{sec:statement-of-results}, we rely on proving multilinear estimates within the framework of Bourgain spaces that lead to well-posedness. To aid us in proving these multilinear estimates we will make heavy use of linear and bilinear smoothing estimates that were derived by the author in the context of the NLS hierarchy equations~\cite{Adams2024}. The multilinear $\hat X^r_{s,b}$ and $X_{s,b}^p$ estimates that lead to well-posedness in~\cite{Adams2024} will also be of use.

We cite the necessary estimates in the following subsection for the reader's convenience and to keep this work more self-contained.

\subsection{Smoothing and multilinear estimates}

To keep in line with how the estimates are stated in~\cite{Adams2024} we introduce the following notational convenience in this subsection: $u$, $v$ and $w$ will refer to functions in appropriate variants of Bourgain spaces adapted to a particular (linear part of an) equation and data spaces at hand so that the right hand side of the respective estimates remain finite. Keeping with the variable choice of the preceding sections $2j$, for $j \in \N$, will be the power in the phase function of the linear equation with which the estimates are associated.

We begin by stating linear estimates based on Kato smoothing and a maximal-function estimate.
\begin{proposition}[\cite{Adams2024}*{Proposition~4.1}]
    Let $b > \1{2}$, then the following inequalities hold
    \begin{enumerate}
        \item for $2 \le q \le \infty$ and $\sigma > \1{2} - \frac{2j}{q}$
        \begin{equation}\label{eq:kato1}
        \norm{u}_{L^\infty_xL^q_t} \lesssim \norm{u}_{X_{\sigma, b}}
        \end{equation}
        \item for $2 \le p \le \infty$ and $\sigma = - \frac{2j - 1}{2}(1 - \frac{2}{p})$
        \begin{equation}\label{eq:kato2}
        \norm{u}_{L^p_xL^2_t} \lesssim \norm{u}_{X_{\sigma, b}}
        \end{equation}
        \item for $4 \le p \le \infty$ and $\sigma >  \1{2} - \1{p}$
        \begin{equation}\label{eq:max-fct}
        \norm{u}_{L^p_xL^\infty_t} \lesssim \norm{u}_{X_{\sigma, b}}
        \end{equation}
    \end{enumerate}
\end{proposition}

In addition we will be making use of a Strichartz-type estimate that is more adapted (and thus more useful) to our Fourier-Lebesgue space setting, referred to most often in the literature as a Fefferman-Stein estimate. In the $L^2$-based setting it reduces to the well-known $L^6$-Strichartz estimate for (higher-order) Schrödinger equations.
\begin{proposition}[\cite{Adams2024}*{Corollary~4.6}]
    Let $j \ge 1$, $0 \le \1{r} < \frac{3}{4}$ and $b > \1{r}$, then one has the estimate
    \begin{equation}\label{eq:fefferman-stein}
        \norm{I^{\frac{2(j-1)}{3r}}u}_{L^{3r}_{xt}} \lesssim \norm{u}_{\hat X^r_{0,b}}.
    \end{equation}
\end{proposition}

Moving on, we may now recall the pair of bilinear operators introduced in~\cite{Adams2024}: For $j \in \N$ and $1 \le p \le \infty$ define $I_{p, j}^{\pm}$ by its Fourier transform
\begin{equation}
\F_x I_{p,j}^\pm(f,g)(\xi) = c \int_* k_j^\pm(\xi_1, \xi_2)^\1{p} \hat{f}(\xi_1)\hat{g}(\xi_2) \d{\xi_1}
\end{equation}
where the symbol is given by
\begin{equation}
k_j^\pm(\xi_1,\xi_2) = |\xi_1 \pm \xi_2|(|\xi_1|^{2j-2} + |\xi_2|^{2j-2}).
\end{equation}
We may now state the $X_{s,b}$ variant of a bilinear estimate involving our bilinear operator(s). To state the proposition we make use of the Fourier-Lebesgue space norms $\norm{f}_{\widehat{L^p}} = \norm{\hat{f}}_{L^{p'}}$.
\begin{proposition}[\cite{Adams2024}*{Corollary~4.3}]
    Let $1 \le q \le r_{1,2} \le p < \infty$ and $b_i > \1{r_i}$. Then we have
    \begin{equation}\label{eq:bilin-est}
    \norm{I_{p,j}^\pm(u,v_\pm)}_{\widehat{L^q_x}\widehat{L^p_t}} \lesssim \norm{u}_{\hat X_{0,b_1}^{r_1}} \norm{v}_{\hat X_{0,b_2}^{r_2}}
    \end{equation}
    where $v_+ = \overline{v}$ and $v_- = v$.
\end{proposition}

Interpreting this bilinear operator as a multiplication operator we may find its adjoint (see~\cite{Adams2024}*{Section~4.2} for details) and gain an additional set of bilinear estimates associated with the adjoint. Let $I_{p,j}^{\pm,*}$ denote this adjoint. It has the symbol
\begin{align}
&k_j^{+,*}(\xi_1, \xi_2) = |\xi_1|(|\xi_1|^{2j-2} + |\xi_2|^{2j-2}), \quad\text{or}\\ &k_j^{-,*}(\xi_1, \xi_2) = |\xi_1 + 2\xi_2|(|\xi_1|^{2j-2} + |\xi_2|^{2j-2}).
\end{align}
We have the following $X_{s,b}$ estimates regarding $I_{p,j}^{\pm,*}$:
\begin{proposition}[\cite{Adams2024}*{Corollary~4.4}]
    Let $1 < q \le r_{1,2} \le p < \infty$ with $\1{p} + \1{q} = \1{r_1} + \1{r_2}$ and $b_i > \1{r_i}$. Then the estimate
    \begin{equation}\label{eq:dual-bilin-free-params}
    \norm{I^{\pm,*}_{p,j}(u,v_\mp)}_{\hat X_{0,-b_1}^{r_1^\prime}} \lesssim \norm{u}_{\widehat{L^{q^\prime}_x}\widehat{L^{p^\prime}_t}} \norm{v}_{\hat X_{0,b_2}^{r_2}}
    \end{equation}
    holds. If alternatively $0 \le \1{\rho^\prime} \le \1{r^\prime}$ and $\beta < -\1{\rho^\prime}$ we have
    \begin{equation}\label{eq:bilin-dual}
    \norm{I^{\pm,*}_{\rho^\prime,j}(u,v_\mp)}_{\hat X_{0,\beta}^r} \lesssim \norm{u}_{\widehat{L^r_{xt}}} \norm{v}_{\hat X_{0,-\beta}^{\rho^\prime}}.
    \end{equation}
    In both cases $v_+ = \overline{v}$ and $v_- = v$.
\end{proposition}

Finally we will later also make use of the trilinear $X_{s,b}$ estimates that leads to well-posedness in Fourier-Lebesgue and/or modulation spaces. Recall:
\begin{proposition}[\cite{Adams2024}*{Proposition~5.1}]\label{prop:nls-cubic-hat-estimate}
    Let $1 < r \le 2$, $s = \frac{j-1}{r'}$, $\alpha \in \N_0^3$ with $|\alpha| = 2(j-1)$. Then there exist $b' < 0$ and $b' + 1 > b > \1{r}$ such that one has
    \begin{equation}\label{eq:nls-cubic-hat-estimate}
    \norm{\partial_x^{\alpha_1}u_1 \partial_x^{\alpha_2}\overline{u_2} \partial_x^{\alpha_3}u_3}_{\hat{X}_{s,b^\prime}^r} \lesssim \prod_{i=1}^{3} \norm{u_i}_{\hat{X}_{s,b}^r}.
    \end{equation}
\end{proposition}

\begin{proposition}[\cite{Adams2024}*{Proposition~5.6}]\label{prop:nls-cubic-modulation-estimate}
    Let $j \ge 2$, $2 \le p < \infty$, $s = \frac{j-1}{2}$, $\alpha \in \N_0^3$ with $|\alpha| = 2(j-1)$. Then there exist $b' < 0$ and $b' + 1 > b > \1{2}$ such that one has
    \begin{equation}\label{eq:nls-cubic-modulation-estimate}
    \norm{\partial_x^{\alpha_1}u_1 \partial_x^{\alpha_2}\overline{u_2} \partial_x^{\alpha_3}u_3}_{X_{s,b^\prime}^p} \lesssim \prod_{i=1}^{3} \norm{u_i}_{X_{s,b}^p}.
    \end{equation}
\end{proposition}

\subsection{Basic estimate on the resonance relation}

As mentioned in the introduction, the additional derivative in the nonlinear terms of dNLS hierarchy equations adds difficulty (over the NLS hierarchy equations) in their analysis. Additional arguments are necessary to overcome this difficulty. The first step in this direction was the introduction and use of the gauge-transformation in order to simplify, or more precisely, remove ill-behaved terms from, the equations. See Section~\ref{sec:gauge-transformation}.

The second step we take in tackling well-posedness estimates for the dNLS hierarchy equations is exploiting the resonance relation, the effectiveness of which was already demonstrated in~\cite{AG2005}. In the absence of an analogue of the exact factorisation for the resonance relation for higher-order dNLS hierarchy equations one may still recover the essence:
\begin{lemma}[\cite{Forlano2019}*{Lemma~2.3}]\label{lem:modulation-estimate}
    Let $\alpha > 1$, $\xi_1, \xi_2, \xi_3 \in \R$ and set $\xi = \xi_1 + \xi_2 + \xi_3$. Then one has
    \begin{equation}
        ||\xi|^\alpha - |\xi_1|^\alpha + |\xi_2|^\alpha - |\xi_3|^\alpha | \gtrsim |\xi_1 + \xi_2| |\xi_2 + \xi_3| |\xi_{max}|^{\alpha - 2},
    \end{equation}
    where $\xi_{max} = \max\set{|\xi_1|, |\xi_2|, |\xi_3|, |\xi|}$.
\end{lemma}
Using this estimate, in combination with the flexibility $X_{s,b}$ spaces offer, will suffice in order to derive the multilinear estimates that lead to well-posedness we are after.
\section{Estimates leading to well-posedness}\label{sec:multilinear-estimates}

With all necessary smoothing estimates that we will need recalled, as well as previous $X_{s,b}$ estimates that we will want to make use of, we are ready to prove the propositions that serve as proof of our Theorem~\ref{thm:gauged-wellposed-hat} and \ref{thm:gauged-wellposed-modulation}. The discussion of the gauge-transformation in Section~\ref{sec:gauge-transformation} combined with these Theorems then also suffice to argue the validity of Theorem~\ref{thm:hierarchy-wellposed-hat-spaces} and~\ref{thm:hierarchy-wellposed-modulation-spaces}.

As is the case for the NLS hierarchy equations, cubic nonlinear terms are more difficult to deal with than their quintic and higher-order counterparts. So we will be dealing with cubic and higher-order terms separately.

\subsection{Multilinear estimates in $\hat X^r_{s,b}$ spaces}
\begin{proposition}\label{prop:cubic-estimate-hat}
    Let $j \ge 2$, $1 < r \le 2$, $s \ge \1{2} + \frac{j-1}{r'}$, and $$\alpha \in \set{(\alpha_1, \alpha_2, \alpha_3) \in \N_0^3 \mid \alpha_1 + \alpha_2 + \alpha_3 = 2j-1,\, \alpha_2 \not=0},$$ then there exist $b' < 0 < \frac{1}{r} < b < b' + 1$ such that the following estimate holds:
    \begin{equation}
    \norm{\partial_x^{\alpha_1}u_1 \partial_x^{\alpha_2}\overline{u_2} \partial_x^{\alpha_3}u_3}_{\hat X_{s, b'}^r} \lesssim \prod_{i=1}^3 \norm{u_i}_{\hat X_{s, b}^r}.
    \end{equation}
\end{proposition}
\begin{proof}
    It suffices to prove this estimate for $s = \1{2} + \frac{j-1}{r'}$ fixed.
    For the proof we want to rely, for the most difficult frequency constellations, on the cubic estimate in Proposition~\ref{prop:nls-cubic-hat-estimate}, which was proven in the author's previous work on the NLS hierarchy~\cite{Adams2024}. Relying on the `equivalent' NLS estimate to prove well-posedness for dNLS was already a successful technique employed in~\cite{AG2005}. Though in addition to the arguments presented there we have to utilise the full gain of the modulation in order to close the estimate.

    In particular, one is able to re-use the NLS estimate, if the frequencies $\xi$, $\xi_1$, $\xi_2$, $\xi_3$ allow for the following inequality:
    \begin{equation}\label{eq:cubic-nice-ineq-hat}
    \JBX[\xi]^s |\xi_1|^{\alpha_1} |\xi_2|^{\alpha_2} |\xi_3|^{\alpha_3} \lesssim \JBX[\xi]^{s-\1{2}} \JBX[\xi_1]^{\alpha_1+\1{2}} \JBX[\xi_2]^{\alpha_2-\1{2}} \JBX[\xi_3]^{\alpha_3+\1{2}}.
    \end{equation}
    This is also where it becomes relevant that we assume that at least a single derivative falls on $\overline{u_2}$. Otherwise $\alpha_2 - \1{2}$ may become negative which would in turn require far more detailed analysis, since the NLS estimate could not be as easily applied.

    Furthermore, from here on we will assume, by symmetry, that $u_1$ has larger frequency than $u_3$ and the largest frequency of $|\xi_1|$, $|\xi_2|$, and $|\xi_3|$ shall synonymously be known as $|\xi_{max}|$.

    \begin{enumerate}[wide]\itemsep1em
        \item \eqref{eq:cubic-nice-ineq-hat} holds. In this case we may use the inequality~\eqref{eq:cubic-nice-ineq-hat} and `reinterpret' the cubic nonlinearity as one how it would appear in an NLS hierarchy equation:
        \begin{align}
        \norm{\partial_x^{\alpha_1}u_1 \partial_x^{\alpha_2}\overline{u_2} \partial_x^{\alpha_3}u_3}_{\hat X_{s, b'}^r} &\lesssim
        \norm{\partial_x^{\alpha_1}(J^\1{2}u_1) \partial_x^{\alpha_2-1}(J^\1{2}\overline{u_2}) \partial_x^{\alpha_3}(J^\1{2}u_3)}_{\hat X_{s-\1{2}, b'}^r}\\
        &\lesssim \prod_{i=1}^{3} \norm{J^\1{2}u_i}_{\hat X_{s-\1{2}, b}^r} \lesssim
        \prod_{i=1}^{3} \norm{u_i}_{\hat X_{s, b}^r},
        \end{align}
        where we were then immediately able to apply~\eqref{eq:nls-cubic-hat-estimate} and arrive at our desired upper bound.

        But when is~\eqref{eq:cubic-nice-ineq-hat} true, i.e. which other cases do we still have to deal with?
        \begin{itemize}
            \item It certainly holds if $|\xi_2| \JBX[\xi] \lesssim \JBX[\xi_1] \JBX[\xi_3]$, as from this~\eqref{eq:cubic-nice-ineq-hat} is quite immediate.
            \item When $|\xi_2| \lesssim 1$ or $|\xi| \lesssim 1$ then~\eqref{eq:cubic-nice-ineq-hat} must also hold. This is because either the frequency $|\xi_2|$ is negligible and can easily be traded against $\xi_1$ or $\xi_3$, or because there exist at least two high-frequency factors between which derivatives can be traded painlessly.
        \end{itemize}

        So in all other cases that follow this one we may assume, without loss of generality, that $|\xi| \sim \JBX[\xi]$, $|\xi_2| \sim \JBX[\xi_2]$, and $\JBX[\xi] \JBX[\xi_2] \gg \JBX[\xi_1] \JBX[\xi_3]$, and we will do so without further mention. Further we will also be showcasing the estimate on condition that the modulation of the product is maximal $\JBX[\sigma_0] = \JBX[\sigma_{max}]$. Cases where the modulations of individual factors are maximal can be proven analogously since the remaining cases are non- or at most semi-resonant.
        \item $|\xi_1| = |\xi_{max}|$. In this case, because of $\JBX[\xi] \JBX[\xi_2] \gg \JBX[\xi_1] \JBX[\xi_3]$, it must hold that $\JBX[\xi] \gg \JBX[\xi_3]$ which in turn implies $\JBX[\xi] \lesssim \JBX[\xi_1 + \xi_2]$. From this we may also derive
        \begin{equation}\label{eq:hat-freq-est-first-case}
        |\xi\xi_2| \lesssim |(\xi_1 + \xi_2)\xi_2| \lesssim |\xi_{max}| |\xi_1+\xi_2| = |\xi_1(\xi_1+\xi_2)|.
        \end{equation}
        Further, using our general estimate for the modulation Lemma~\ref{lem:modulation-estimate} we have
        \begin{equation}\label{eq:hat-modulation-est}
        \JBX[\sigma_0] \gtrsim |\xi_1 + \xi_2| |\xi_2 + \xi_3| |\xi_{max}|^{2j-2} \gtrsim |\xi\xi_2| |\xi_1|^{2j-2}
        \end{equation}
        at our disposal. With all preparations done we may focus on proving the estimate.

        As our first step we shift all derivatives of the product, except for one guaranteed to lie on $u_2$, to $u_1$ and use our estimate for the modulation~\eqref{eq:hat-modulation-est}.
        \begin{align}
        &\norm{\partial_x^{\alpha_1}u_1 \partial_x^{\alpha_2}\overline{u_2} \partial_x^{\alpha_3}u_3}_{\hat X_{s, b'}^r} \lesssim
        \norm{\Lambda^{b'} J^{s+\1{2}}((I^{|\alpha|-\frac{3}{2}}u_1) (I^\frac{3}{2}\overline{u_2}) (I^{-\1{2}}u_3))}_{\widehat{L^r_{xt}}}\\
        &\lesssim \norm{J^{s+\1{2}-\1{r'}+}((I^{|\alpha|-\frac{3}{2}-\frac{2j-2}{r'}+}u_1) (I^{\frac{3}{2}-\1{r'}+}\overline{u_2}) (I^{-\1{2}}u_3))}_{\widehat{L^r_{xt}}}
        \intertext{Now using~\eqref{eq:hat-freq-est-first-case} we may shift derivatives again to arrive at:}
        &\lesssim \norm{J^{s+\1{2}-\1{r'}-\1{r}+}(I^\1{r}((I^{|\alpha|-\frac{3}{2}-\frac{2j-2}{r'}+\1{r}+}u_1) (I^{\frac{3}{2}-\1{r'}-\1{r}+}\overline{u_2})) (I^{-\1{2}}u_3))}_{\widehat{L^r_{xt}}},
        \intertext{where we are now ready to upgrade $I^\1{r}$ to our well-known bilinear operator $I^+_{r,j}$ and then apply its corresponding estimate, after dealing with $u_3$ by Hölder's inequality.}
        &\lesssim \norm{J^{s-\1{2}+}(I^\1{r}((I^{|\alpha|-\frac{3}{2}-\frac{2j-2}{r'}+\1{r}+}u_1) (I^{\1{2}+}\overline{u_2})) (I^{-\1{2}}u_3))}_{\widehat{L^r_{xt}}}\\
        &\lesssim \norm{I^+_{r,j}(I^{|\alpha|+s-2-(2j-2)+\1{r}+}u_1, I^{\1{2}+}\overline{u_2})}_{\widehat{L^r_{xt}}} \norm{I^{-\1{2}}u_3}_{\widehat{L^\infty_{xt}}}\\
        &\lesssim \norm{J^{s-1+\1{r}+}u_1}_{\hat X_{0,b}^r} \norm{I^{\1{2}+}u_2}_{\hat X_{0, b}^r} \norm{I^{\1{r}-\1{2}+}u_3}_{\hat X_{0, b}^r} \lesssim \prod_{i=1}^{3} \norm{u_i}_{\hat X^r_{s,b}}
        \end{align}
        The last inequality holds, so long as $\1{r}-1 < 0$ and $s \ge \1{2}+\frac{j-1}{r'} \ge \max(\1{2}+, \1{r}-\1{2}+)$, which is the case for $r > 1$.

        \item $|\xi_2| = |\xi_{max}|$. The argument in this case is quite similar to the preceding case, only that now we do not have to account for the guaranteed derivative on $u_2$ as this is the high-frequency factor to which we shift all derivatives anyway.

        When $|\xi_2|$ is maximal it follows that $|\xi_2| \sim |\xi| \sim |\xi_1 + \xi_2| \sim |\xi_2 + \xi_3|$ and for the modulation, again using Lemma~\ref{lem:modulation-estimate}, we can estimate $\JBX[\sigma_0] \gtrsim |\xi_1+\xi_2| |\xi_2+\xi_3| |\xi_{max}|^{2j-2} \gtrsim |\xi_2|^{2j}$. For proving our estimate this leads us to
        \begin{align}
        &\norm{\partial_x^{\alpha_1}u_1 \partial_x^{\alpha_2}\overline{u_2} \partial_x^{\alpha_3}u_3}_{\hat X_{s, b'}^r} \lesssim
        \norm{\Lambda^{b'} (u_1 (I^{|\alpha|+s} \overline{u_2}) u_3)}_{\widehat{L^r_{xt}}}\\
        &\lesssim \norm{(I^{-\1{2}}u_1) I^{\1{r}}((I^{|\alpha|+s+\1{2}-\1{r}-\frac{2j}{r'}+} \overline{u_2}) u_3)}_{\widehat{L^r_{xt}}}
        \\&\lesssim \norm{I^{-\1{2}}u_1}_{\widehat{L^\infty_{xt}}} \norm{I^{\1{r}}((I^{|\alpha|+s+\1{2}-\1{r}-\frac{2j}{r'}+} \overline{u_2}) u_3)}_{\widehat{L^r_{xt}}}\notag\\
        &\lesssim \norm{I^{\1{r}-\1{2}+}u_1}_{\hat X_{0, b}^r} \norm{I^+_{r,j}(I^{|\alpha|+s+\1{2}-\1{r}-\frac{2j}{r'}-\frac{2j}{r}+} \overline{u_2}, u_3)}_{\widehat{L^r_{xt}}}\\
        &\lesssim \norm{I^{\1{r}-\1{2}+}u_1}_{\hat X_{0, b}^r} \norm{I^{s-\1{2}-\1{r}+} \overline{u_2}}_{\hat X^r_{0, b}} \norm{u_3}_{\hat X^r_{0,b}} \lesssim \prod_{i=1}^{3} \norm{u_i}_{\hat X^r_{s,b}},
        \end{align}
        which is the desired upper bound, if $r > 1$, so the proof is complete.
    \end{enumerate}
\end{proof}

Since in the proof of the necessary quintilinear (and higher-order) estimate to argue our well-posedness Theorems we rely on the fact $s < \1{r}$, we will argue the estimate for the full range of parameter $1 < r \le 2$ in two parts. First we will prove Proposition~\ref{prop:quintic-estimate-hat-endpoint} below, which for a comparatively higher level of regularity establishes the multilinear estimate near the endpoint $r \to 1$. This we can then in turn interpolate with the $L^2$-based estimate that is part of Proposition~\ref{prop:quintic-estimate-modulation} in order to cover the full parameter range.
\begin{proposition}\label{prop:quintic-estimate-hat-endpoint}
    Let $j \ge 2$, $2 \le k \le 2j$ and $\alpha \in \N_0^{2k+1}$ with $|\alpha| = 2j-k$. Then there exists an $1 < r_0 \ll 2$ such that for all $1 < r < r_0$ and $s > \1{2}+\frac{j-k}{kr'}$ there exist $b' < 0 < \frac{1}{r} < b < b' + 1$ such that the following estimate holds
    \begin{equation}
    \norm{\prod_{i=1}^{2k+1} \partial_x^{\alpha_{i}}v_{i}}_{\hat X_{s, b'}^r} \lesssim \prod_{i=1}^{2k+1} \norm{u_i}_{\hat X_{s, b}^r},
    \end{equation}
    where exactly $k$ of the factors $v_1, v_2, \ldots, v_{2k+1}$ are equal to $\overline{u_i}$ and otherwise just equal to $u_i$.
\end{proposition}
\begin{proof}
    Without loss of generality we may assume that the frequencies of the $2k+1$ factors are ordered decreasingly, i.e. $|\xi_1| \ge |\xi_2| \ge \ldots |\xi_{2k+1}|$. We will analyse the product based on the number of high-frequency factors present.

    Throughout the proof we will need to make use of the fact $s - \1{r} < 0$, which we may achieve by choosing $1 < r_0 \ll 2$ appropriately small.
    \begin{enumerate}[wide]\itemsep1em
        \item $|\xi_4| \gtrsim |\xi_1|$, so we have at least 4 high-frequency factors. In this case every factor of the product passes through a norm that is invariant with respect to complex conjugation, so we may ignore its distribution among the factors in this case.

        The idea of the proof in this case is to use the Fefferman-Stein estimate for the high-frequency factors (of which we need 4 in order to ensure its applicability) and a Sobolev-type embedding for the rest. We start by distributing the derivatives of the norm and those in the product on the high-frequency factors in addition to leaving a little leeway for embeddings on the remaining factors. Then we use the Hausdorff-Young inequality to 'remove the hat' from the space.
        \begin{align}
        &\norm{\prod_{i=1}^{2k+1} \partial_x^{\alpha_{i}}v_{i}}_{\hat X_{s, b'}^r}
        \lesssim \norm{(J^{\sigma}u_1)(J^{\sigma}u_2)(J^{\sigma}u_3)(J^{\sigma}u_4)\prod_{i=5}^{2k+1} J^{s-\1{r}-}u_i}_{L^r_{xt}}
        \intertext{This requires $\sigma \ge 0$, $4\sigma + (2k-3)(s-\1{r}) \ge s + 2j-k$ as well as $s -\1{r} - < 0$ the latter of which is ensured by our choice of $r_0$ at the beginning of this proof. By now using Hölder's inequality, Hausdorff-Young again (to put the hat back on $L^\infty$) and a Sobolev-type embedding we arrive at:}
        &\lesssim \prod_{i=1}^{4}\norm{J^{\sigma}u_i}_{L^{4r}_{xt}} \prod_{i=5}^{2k+1} \norm{J^{s-\1{r}-}u_i}_{\widehat{L^\infty_{xt}}} \lesssim \prod_{i=1}^{4}\norm{J^{\sigma}u_i}_{L^{4r}_{xt}} \prod_{i=5}^{2k+1} \norm{u_i}_{\hat X_{s, b}^r}
        \intertext{As announced before, we may now use the Fefferman-Stein estimate~\eqref{eq:fefferman-stein} which grants us a gain of $\frac{2(j-1)}{4r}$ derivatives on each of the high-frequency factors, but leaves us in the wrong $\hat X_{0,\frac{3}{4r}+}^{\frac{4r}{3}}$ space. To remedy this we may use a Sobolev-type inequality for which we have to spend $\1{4r}+$ derivatives.}
        &\lesssim \prod_{i=1}^{4}\norm{J^{\sigma - \frac{2(j-1)}{4r}+\1{4r}+}u_i}_{\hat X_{0,b}^r} \prod_{i=5}^{2k+1} \norm{u_i}_{\hat X_{s, b}^r} \lesssim \prod_{i=1}^{2k+1} \norm{u_i}_{\hat X_{s, b}^r}
        \end{align}
        The reader my verify that for $s > \1{2} + \frac{j-k}{kr'}$ and the choice $\sigma = s + \frac{2(j-1)}{4r} - \1{4r}-$ the requirements gathered involving $\sigma$ can be fulfilled and we may justify the final inequality to arrive at the desired upper bound.

        \item $|\xi| \sim |\xi_1| \gg |\xi_2|$, so we have only a single high-frequency factor.

        One needs to take care as to what the distribution of complex conjugates in the product is. We will showcase a proof of the estimate in the instance that the product we are dealing with is equal to $\overline{u_1} (\prod_{i=2}^{2k-3}v_i) u_{2k-2}u_{2k-1}u_{2k}\overline{u_{2k+1}}$ (ignoring derivatives). This aligns with the requirement, that $k$ of the factors are complex conjugates. The other cases, for different distributions of complex conjugates can be dealt with in a similar fashion and we omit the details.

        With only a single large frequency we have immediate control over the symbols of the bilinear operators $I^+_{r,j}$ and $I^{-,*}_{\rho',j}$. We proceed by shifting all derivatives of the norm and in the product onto the high-frequency factor in addition to some extra derivatives we will later need for Sobolev-type embeddings.
        \begin{align}
        \norm{\prod_{i=1}^{2k+1} \partial_x^{\alpha_{i}}v_{i}}_{\hat X_{s, b'}^r}
        \lesssim \norm{J^{\sigma}\overline{u_1}(\prod_{i=2}^{2k-1} J^{s-\1{r}-} v_{i}) J^{s}u_{2k} J^{s-\1{r}+\1{\rho'}-}\overline{u_{2k+1}}}_{\hat X_{0, b'}^r}
        \end{align}
        Here we have introduced $\sigma \ge 0$ and this inequality holds so long as $\sigma + (2k-2)(s-\1{r}) + s + (s-\1{r}+\1{\rho'}) > s + 2j-k$. Furthermore $s -\1{r}- < 0$ is ensured by our choice of $r_0$ at the beginning of this proof.

        Next we introduce the bilinear operator $I^+_{r,j}$ which grants us a gain of $\frac{2j-1}{r}$ derivatives on the high-frequency factor:
        \begin{align}
        &\lesssim \norm{I^+_{r,j}(J^{\sigma-\frac{2j-1}{r}}\overline{u_1}, J^{s}u_{2k})(\prod_{i=2}^{2k-1} J^{s-\1{r}-} v_{i}) J^{s-\1{r}+\1{\rho'}-}\overline{u_{2k+1}}}_{\hat X_{0, b'}^r}\\
        &\lesssim \norm{I^{-,*}_{\rho',j}(I^+_{r,j}(J^{\sigma-\frac{2j-1}{r}-\frac{2j-1}{\rho'}}\overline{u_1}, J^{s}u_{2k})\prod_{i=2}^{2k-1} J^{s-\1{r}-} v_{i}, J^{s-\1{r}+\1{\rho'}-}\overline{u_{2k+1}})}_{\hat X_{0, b'}^r}
        \intertext{Choosing $\1{\rho'} \le \1{r}$ with $\frac{2(j-k)}{r} < \frac{2j}{\rho'}$ we now also introduce its dual $I^{-,*}_{\rho',j}$, which grants us $\frac{2j-1}{\rho'}$ on the high-frequency factor. Now applying the continuity of the dual bilinear operator~\eqref{eq:dual-bilin-free-params} and Hölder's inequality we may derive}
        &\lesssim \norm{I^+_{r,j}(J^{\sigma-\frac{2j-1}{r}-\frac{2j-1}{\rho'}}\overline{u_1}, J^{s}u_{2k})\prod_{i=2}^{2k-1} J^{s-\1{r}-} v_{i}}_{\widehat{L^r_{xt}}} \norm{J^{s-\1{r}+\1{\rho'}-}u_{2k+1}}_{\hat X_{0, -b'}^{\rho'}}\\
        &\lesssim \norm{I^+_{r,j}(J^{\sigma-\frac{2j-1}{r}-\frac{2j-1}{\rho'}}\overline{u_1}, J^{s}u_{2k})}_{\widehat{L^r_{xt}}} \prod_{i=2}^{2k-1} \norm{J^{s-\1{r}-} u_{i}}_{\widehat{L^\infty_{xt}}} \norm{J^{s-\1{r}+\1{\rho'}-}u_{2k+1}}_{\hat X_{0, -b'}^{\rho'}}\notag
        \intertext{For the first factor we apply the continutiy of the bilinear operator~\eqref{eq:bilin-est}, for the factors in the product we use a Sobolev-type embedding and for the final factor Young's inequality:}
        &\lesssim \norm{J^{\sigma-\frac{2j-1}{r}-\frac{2j-1}{\rho'}}u_1}_{\hat X_{0, b}^r} \norm{J^s u_{2k}}_{\hat X_{0, b}^r} \prod_{i=2}^{2k-1} \norm{J^s u_{i}}_{\hat X_{0, b}^r} \norm{J^s u_{2k+1}}_{\hat X_{0, b}^{r}}
        \end{align}
        By choosing $\sigma = \frac{2k-1}{r} + 2j-k -(2k-1)s - \1{\rho'} + > 0$ and our choice of $\rho'$ the reader may verify that $\sigma-\frac{2j-1}{r}-\frac{2j-1}{\rho'} < s$ and that the other requirements with respect to $\sigma$ are fulfilled. Hence we have accomplished the proof of the estimate in this case.

        \item $|\xi_2| \gtrsim |\xi_1| \gg |\xi_3|$ or $|\xi_3| \gtrsim |\xi_1| \gg |\xi_4|$, so we have two or three high-frequency factors. In this case, also depending on which factors are complex conjugates, we may parenthesise differently in use of the bilinear operator to the preceding case. Different distributions of complex conjugates may be dealt with by using either $I^+_{r, j}$ or $I^-_{r, j}$ (and their duals) appropriately. The arguments are similar to case we have already dealt with, so we choose to omit the details.
    \end{enumerate}
\end{proof}

\subsection{Multilinear estimates in $X^p_{s,b}$ spaces}
The proof of the cubic estimate in modulation space-based $X_{s,b}$ spaces is very similar to the proof of Proposition~\ref{prop:cubic-estimate-hat} (in the $r=2$ case), where the equivalent estimate for Fourier-Lebesgue-based spaces is showcased. We choose to omit the details that are analogous and only show the necessary additional arguments.
\begin{proposition}\label{prop:cubic-estimate-modulation}
    Let $j \ge 2$, $2 \le p < \infty$, $s \ge \frac{j}{2}$, and $$\alpha \in \set{(\alpha_1, \alpha_2, \alpha_3) \in \N_0^3 \mid \alpha_1 + \alpha_2 + \alpha_3 = 2j-1,\, \alpha_2 \not=0},$$ then there exist $b' < 0 < \frac{1}{r} < b < b' + 1$ such that one has the estimate
    \begin{equation}
    \norm{\partial_x^{\alpha_1}u_1 \partial_x^{\alpha_2}\overline{u_2} \partial_x^{\alpha_3}u_3}_{X_{s, b'}^p} \lesssim \prod_{i=1}^3 \norm{u_i}_{X_{s, b}^p}.
    \end{equation}
\end{proposition}
\begin{proof}
Main idea of the proof is again to reuse the corresponding NLS estimate for cubic terms (that is Proposition~\ref{prop:nls-cubic-modulation-estimate} in this case) in the most difficult resonant cases. We argue along the lines of the first case in the proof of Proposition~\ref{prop:cubic-estimate-hat}, replacing any mention of an $\hat X^r_{s,b}$ space with the appropriate $X^p_{s,b}$ space.

What is left is to argue the remaining two cases where either $\xi_1$ or $\xi_2$ is the maximal frequency (here we have also adopted the convention that the frequency of $u_1$ is greater than that of $u_3$, without loss).

For both cases we begin by using the trivial embedding $X_{s,b'}^p \supset X_{s,b'}$, so that we may reuse what was argued in the $r=2$ case in Proposition~\ref{prop:cubic-estimate-hat}. Following along the lines of the proof one arrives at a bound
\begin{align}
    \norm{\partial_x^{\alpha_1}u_1 \partial_x^{\alpha_2}\overline{u_2} \partial_x^{\alpha_3}u_3}_{X_{s, b'}^p} &\lesssim
    \norm{\partial_x^{\alpha_1}u_1 \partial_x^{\alpha_2}\overline{u_2} \partial_x^{\alpha_3}u_3}_{X_{s, b'}}\\
    &\lesssim \norm{J^{s-\1{2}+}u_1}_{\hat X_{0,b}} \norm{u_2}_{\hat X_{0+, b}} \norm{u_3}_{\hat X_{\1{2}+, b}},
\end{align}
where possibly the roles of $u_1$, $u_2$ and $u_3$ are interchanged depending on the exact case (i.e. $|\xi_1| = |\xi_{max}|$ or $|\xi_2| = |\xi_{max}|$). Now, using the Sobolev-type embedding for modulation spaces~\eqref{eq:modulation-sobolev}, we may bound this by our desired right-hand side so long as $s -\1{2}+\1{2}-\1{p} < s$ and $\1{2}+\1{2}-\1{p} < s$, which can be achieved for $p < \infty$, $j \ge 2$ and $s \ge \frac{j}{2}$ as claimed.
\end{proof}

\begin{proposition}\label{prop:quintic-estimate-modulation}
    Let $j \ge 2$, $2 \le p \le \infty$, $2 \le k \le 2j$, $s > \1{2} + \1{4k} - \frac{2k+1}{2kp}$ and $\alpha \in \N_0^{2k+1}$ with $|\alpha| = 2j-k$, then there exist $b' < 0 < \frac{1}{2} < b < b' + 1$ such that the following estimate holds:
    \begin{equation}\label{eq:quintic-estimate-modulation}
    \norm{\prod_{i=1}^{2k+1} \partial_x^{\alpha_i}u_i}_{X_{s, b'}^p} \lesssim \prod_{i=1}^{2k+1} \norm{u_i}_{X_{s, b}^p}.
    \end{equation}
    Additionally, the distribution of complex conjugates on the factors $u_i$ may be chosen arbitrarily.
\end{proposition}
\begin{proof}
We will assume, without loss of generality by symmetry, that the frequencies of the factors in the product are order decreasingly, i.e. $|\xi_1| \ge |\xi_2| \ge \ldots \ge |\xi_{2k+1}|$. Depending on if the largest frequency of one of the factors is comparable (or not) to the frequency of the product we differentiate between two cases.

The reader may note that in both cases each factor passes through a mixed $L^p_xL^q_t$ which is invariant with respect to complex conjugation. This justifies the addition to the theorem, that the distribution of complex conjugates may be chosen arbitrarily.

Idea of the proof is to use reduce the proof to the $L^2$ case, where Kato smoothing for two of the `factors' and the maximal function estimate for the rest is used, and to then use a Sobolev-type embedding to get back to the correct modulation space. The latter is what leads to the restriction on $s$, i.e. in the $L^2$ case we reach scaling up to an epsilon.
\begin{enumerate}[wide]\itemsep1em
    \item $|\xi| \sim |\xi_1|$. Since $u_1$ is the factor with the largest frequency, comparable with the product itself, we may redistribute all derivatives in the product accordingly. In the same step we use the trivial embedding $X_{s, b'}^p \supset X_{s, b'}^2$, for $p \ge 2$, and introduce $\sigma \ge 0$ to be choosen later as well as $r = \infty-$.
    \begin{align}
        &\norm{\prod_{i=1}^{2k+1} \partial_x^{\alpha_i}u_i}_{X_{s, b'}^p} \lesssim \norm{(J^{s+2j-k - \frac{2j-1}{2}(1-\frac{2}{r}) + \sigma} u_1) \prod_{i=2}^{2k+1} J^{-\frac{\sigma}{2k}}u_i}_{X_{\frac{2j-1}{2}(1-\frac{2}{r}), b'}}
        \intertext{Now we may use the modulation with exponent $b'=-\1{2}+$ by applying the dual version of Kato's smoothing estimate~\eqref{eq:kato2}, followed by an application of Hölder's inequality.}
        &\lesssim \norm{(J^{s+2j-k - \frac{2j-1}{2}(1-\frac{2}{r})+\sigma} u_1) \prod_{i=2}^{2k+1} J^{-\frac{\sigma}{2k}}u_i}_{L^{r'}_x L^2_t}\\
        &\lesssim \norm{J^{s+2j-k - \frac{2j-1}{2}(1-\frac{2}{r})+\sigma} u_1}_{L^\infty_xL^2_t} \prod_{i=2}^{2k+1} \norm{J^{-\frac{\sigma}{2k}}u_i}_{L^{2kr'}_x L^\infty_t}
        \intertext{Another application of Kato's smoothing inequality~\eqref{eq:kato2} for the first factor and the maximal function estimate~\eqref{eq:max-fct} leads us to:}
        &\lesssim \norm{J^{s+2j-k - \frac{2j-1}{2}(2-\frac{2}{r})+\sigma} u_1}_{X_{0,b}} \prod_{i=2}^{2k+1} \norm{J^{-\frac{\sigma}{2k}+\1{2} - \1{2kr'}+}u_i}_{X_{0,b}}\\
        &\lesssim \norm{J^{s+2j-k - \frac{2j-1}{2}(2-\frac{2}{r})+\sigma+\1{2}-\1{p}+} u_1}_{X_{0,b}^p} \prod_{i=2}^{2k+1} \norm{J^{-\frac{\sigma}{2k}+\1{2} - \1{2kr'}+\1{2}-\1{p}+}u_i}_{X_{0,b}^p}
    \end{align}
    where we have applied the Sobolev-type embedding~\eqref{eq:modulation-sobolev} to each of the factors. This product as a whole may be bounded by our desired right hand side in~\eqref{eq:quintic-estimate-modulation} on condition that
    \begin{align}
        &2j-k - \frac{2j-1}{2}(2-\frac{2}{r})+\sigma+\1{2}-\1{p} < 0 \qquad\text{and}\\
        &-\frac{\sigma}{2k}+\1{2} - \1{2kr'}+\1{2}-\1{p} < s.
    \end{align}
    We leave it to the reader to verify that, so long as $s > \1{2} + \1{4k} - \frac{2k+1}{2kp}$, these inequalities hold, if one chooses $\sigma = k - \frac{3}{2} + \1{p}-$ which clearly also fulfils $\sigma \ge 0$.

    \item $|\xi| \ll |\xi_1|$ so that we must have $|\xi_1| \sim |\xi_2|$. The proof in this case is similar in spirit to the preceding case, only that, since the frequency of the product is small, it is more beneficial to apply Kato's smoothing inequality to the first two factors.

    After redistributing derivatives beneficially and moving to $L^2$-based Bourgain spaces as above, we use the modulation of the product (with exponent $b' = -\1{2}+$) for a Sobolev embedding in time. In the space variable we also sacrifice a total of $\1{2}-$ derivatives for a Sobolev embedding to $L^{1+}$.
    \begin{align}
        &\norm{\prod_{i=1}^{2k+1} \partial_x^{\alpha_i}u_i}_{X_{s, b'}^p} \lesssim \norm{(J^{\frac{s}{2} + \sigma_1 + \1{4}-}u_1)(J^{\frac{s}{2} + \sigma_1 + \1{4}-}u_2) \prod_{i=3}^{2k+1} J^{\sigma_2}u_i}_{L^{1+}_{xt}}
        \intertext{Here we have introduced $\sigma_1 \ge 0$ and $\sigma_2 \le 0$ which are to be chosen later under the constraint $2\sigma_1 + (2k-1)\sigma_2 = 2j-k$. Next we may apply Hölder's inequality in preparation for applications of Kato smoothing~\eqref{eq:kato1} for the first two factors. For the remaining factors one has to be careful: Either one can apply the maximal function estimate~\eqref{eq:max-fct} if one has enough factors, that is $k > 3$, or one resorts to using a Sobolev embedding which works just as well for $k=2$ or $k=3$.}
        &\lesssim \norm{J^{\frac{s}{2} + \sigma_1 + \1{4}-}u_1}_{L^\infty_xL^{2+}_t}\norm{J^{\frac{s}{2} + \sigma_1 + \1{4}-}u_2}_{L^\infty_xL^{2+}_t} \prod_{i=3}^{2k+1} \norm{J^{\sigma_2}u_i}_{L^r_xL^\infty_t}\\
        &\lesssim \norm{J^{\frac{s}{2} + \sigma_1 + \1{4}+\1{2} - \frac{2j}{2+}-}u_1}_{X_{0,b}}\norm{J^{\frac{s}{2} + \sigma_1 + \1{4}+\1{2} - \frac{2j}{2+}-}u_2}_{X_{0,b}} \prod_{i=3}^{2k+1} \norm{J^{\sigma_2+\1{2}-\1{r}+}u_i}_{X_{0,b}}\notag
    \end{align}
    Here we have introduced $r$ such that $\frac{2k}{r} = \1{1+} = 1-$ in an intermediate step. This final product may again be bounded by our desired right hand side in~\eqref{eq:quintic-estimate-modulation} after an application of the Sobolev-type embedding for modulation spaces~\eqref{eq:modulation-sobolev}, if the following conditions are met:
    \begin{equation}
        \sigma_1 + \1{4}+\1{2} - \frac{2j}{2} +\1{2} - \1{p} < \frac{s}{2} \qquad\text{and}\qquad \sigma_2+\1{2}-\1{r} + \1{2}-\1{p} < s.
    \end{equation}
    By choosing
    \begin{equation}
        \sigma_1 = j -\frac{3}{4} - \frac{2k-1}{8k} + \frac{2k-1}{4kp} \qquad\text{and}\qquad \sigma_2 = -\1{2} + \1{4k} - \1{2kp} + \1{2k-1}
    \end{equation}
    one may verify that these conditions (and those placed upon $\sigma_1$ and $\sigma_2$) are met so long as $s > \1{2} + \1{4k} - \frac{2k+1}{2kp}$ and thus the proof is complete.
\end{enumerate}
\end{proof}
From Propositions~\ref{prop:cubic-estimate-modulation} and~\ref{prop:quintic-estimate-modulation}, possibly also using the gauge-transformation, the well-posedness Theorems we mentioned at the beginning of this section are now immediate from general theory on $X^p_{s,b}$ spaces. See~\cite{Adams2024}*{Section~1.2} for references on this matter.

\quad

Moving back to estimates in Fourier-Lebesgue-based spaces, we may now use the $L^2$-based (that is $p=2$) estimate that is contained in Proposition~\ref{prop:quintic-estimate-modulation} and interpolate (by the complex multilinear interpolation method) with the near-endpoint estimate from Proposition~\ref{prop:quintic-estimate-hat-endpoint} in order to cover the full parameter range $1 < r \le 2$ that is necessary to argue our well-posedness Theorems in such spaces.
\begin{corollary}\label{cor:quintic-estimate-hat}
    Let $j \ge 2$, $1 < r \le 2$, $2 \le k \le 2j$, $s > \1{r}-\1{2}$ and $\alpha \in \N_0^{2k+1}$ with $|\alpha| = 2j-k$. Then there exist $b' < 0 < \frac{1}{r} < b < b' + 1$ such that the following estimate holds
    \begin{equation}
    \norm{\partial_x^{\alpha_1}u_1\prod_{i=1}^k \partial_x^{\alpha_{2i}}\overline{u_{2i}} \partial_x^{\alpha_{2i+1}}u_{2i+1}}_{\hat X_{s, b'}^r} \lesssim \prod_{i=1}^{2k+1} \norm{u_i}_{\hat X_{s, b}^r}.
    \end{equation}
\end{corollary}
The Theorems mentioned at the beginning of this section regarding well-pos\-ed\-ness in Fourier-Lebesgue spaces may now be derived from Proposition~\ref{prop:cubic-estimate-hat} and Corollary~\ref{cor:quintic-estimate-hat}, possibly in combination with use of the gauge-transformation, with standard theory on $\hat X^r_{s,b}$ spaces. See~\cite{Adams2024}*{Section~1.2} for references on this matter.

\section{Proofs of ill-posedness results}\label{sec:illposedness}

With our well-posedness results established, we now proceed to demonstrate that these results are, in a certain sense, optimal. Specifically, the following arguments will prove Theorems~\ref{thm:illposed-line-C3-hat} and~\ref{thm:illposed-line-C3-modulation}, showing that it is impossible to achieve well-posedness for the equations of interest below the regularity threshold we have already identified using the direct application of the contraction mapping theorem. Additionally, we will show that for periodic initial data, achieving analogous results to those in the nonperiodic case from the previous sections is also unfeasible with the contraction mapping principle.

The argument we use was initially investigated in~\cite{Bourgain1997} and then later refined in~\cite{MR1735881}. By now it has found widespread use to show ill-posedness results for power-type nonlinearities appearing in a wide variety of dispersive equations.

\begin{proof}[Proof of Theorem~\ref{thm:illposed-line-C3-hat}]
    Let us assume that the flow $S : \hat H^s_r(\R) \times (-T, T) \to \hat H^s_r(\R)$ of the Cauchy problem~\eqref{eq:general-cauchy} for a general nonlinearity $N(u)$ containing a cubic term with $2j-1$ derivatives is thrice continuously differentiable. (See the discussion in Remark~\ref{rem:C3-illposedness-line} for what `general nonlinearity' means.) We will as a necessary condition on the regularity of the initial data that $s \ge \1{2} + \frac{j-1}{r'}$.

    For initial datum $u_0(x) = \delta \phi(x)$, where $\delta > 0$ and $\phi \in \hat H^s_r(\R)$ for any $1 \le r \le \infty$ and $s \in \R$ are to be chosen later, we calculate the third derivative of the flow at the origin. Let $u$ denote the solution corresponding to $u_0$ as initial data, then
    \begin{equation}\label{eq:flow-third-derivative}
        \left. \frac{\partial^3 u}{\partial \delta^3} \right|_{\delta = 0} \sim \int_0^t U(t-t') N_3(U(t')u_0)\d{t'},
    \end{equation}
    where we use $U(t)$ to denote the linear propagator of our equation and $N_3(u)$ to refer only to the cubic nonlinear terms in the nonlinearity of our equation. The higher-order nonlinear terms disappear from the third derivative of the flow, because we are evaluating it at $\delta = 0$.

    For our choice of initial data we now introduce parameters $N \gg 1$ and $\gamma \ll 1$ that are to be chosen later. With these in hand we may set $\hat{\phi}(\xi) = \gamma^{-\1{r'}}N^{-s} \chi(\xi)$, where $\chi(\xi)$ is the characteristic function of the interval $[N,N+\gamma]$. The factors in the definition of $\phi$ are chosen such that we have $\norm{\phi}_{\hat H^s_r} \sim 1$.

    Our next step is inserting our initial datum $u_0$ into~\eqref{eq:flow-third-derivative}:
    \begin{align*}
        \F_x \left(\left.\frac{\partial^3 u}{\partial \delta^3} \right|_{\delta = 0}\right)(\xi, t) &\sim \xi^{2j-1} \int_{*} \int_0^t e^{it(-\xi^{2j} + \xi_1^{2j} - \xi_2^{2j} + \xi_3^{2j})} \hat{\phi}(\xi_1) \hat{\overline\phi}(\xi_2)\hat{\phi}(\xi_3) \d{t}\d{\xi_1}\d{\xi_2}
        \intertext{In order to properly bound the inner $t$-integral we must have control of the resonance relation $\Phi = -\xi^{2j} + \xi_1^{2j} - \xi_2^{2j} + \xi_3^{2j}$ of which Lemma~\ref{lem:modulation-estimate} tells us that we may bound it by $\Phi \sim \gamma^2 N^{2j-2}$. Hence we see a choice of $\gamma \sim N^{-(j-1)}$ is sensible. We continue working on a lower bound:}
        \\&\gtrsim t N^{-3s} \gamma^{-\frac{3}{r'}} N^{2j-1} \,\chi*\chi*\chi(\xi)
        \\&\gtrsim t N^{-3s} \gamma^{-\frac{3}{r'}} N^{2j-1} \gamma^2 \,\chi(\xi)
        \\&\sim t N^{-2s+2j-1} \gamma^{2-\frac{2}{r'}} (\gamma^{-\1{r'}} N^{-2} \chi(\xi)) = t N^{-2s+\frac{2j-2}{r'}+1}  \hat{\phi}(\xi).
    \end{align*}
    Here we may now take the the $\hat H^s_r(\R)$ norm of both sides, keeping in mind our choice of $\phi$ leading to $\norm{\phi}_{\hat H^s_r} \sim 1$. Thus we have a lower bound on the third derivative of the flow
    \begin{equation}
        \norm{\left.\frac{\partial^3 u}{\partial \delta^3} \right|_{\delta = 0}}_{\hat H^s_r} \gtrsim t N^{-2s+\frac{2j-2}{r'}+1}.
    \end{equation}
    In order for this quantity to stay bounded (a necessity, if the flow shall be thrice continuously differentiable) we must have $-2s+\frac{2j-2}{r'}+1 \le 0 \iff s \ge \1{2}+\frac{j-1}{r'}$, since otherwise we can let $N \to \infty$ and thus produce a contradiction.
\end{proof}

We will omit the proof of Theorem~\ref{thm:illposed-line-C3-modulation} as it follows along the same lines as the $r=2$ case in the preceding proof. The key insight to be had is, because it suffices the look at the high-high-high interaction, with frequencies located on a single interval of length $o(1)$, the exact choice of Hölder exponents $p,q$ in the modulation spaces is irrelevant. This argument was also given in~\cite{Klaus2023}. Hence the $C^3$ ill-posedness result in modulation spaces parallels the $r=2$ case in Fourier-Lebesgue spaces in terms of regularity ($s < \frac{j}{2}$), but with arbitrary exponents $p,q$.


Having addressed the non-periodic setting, we now present two propositions that establish our ill-posedness results for gauged dNLS equations on the torus. Their proofs follow arguments well-known to the relevant literature and correspond to Theorems~\ref{thm:illposed-torus-C3} and~\ref{thm:illposed-torus-uniform}, respectively.

\begin{proposition}\label{prop:illposed-torus-C3}
    The flow $S : \hat H^s_r(\T) \times (-T, T) \to \hat H^s_r(\T)$ of the Cauchy problem for the fourth-order dNLS hierarchy equation (which corresponds to $j = 2$)
    \begin{equation*}
        i\partial_t u - \partial_x^4 u = \partial_x(-iu^2\overline{u}_{xx} - 4i|u|^2u_{xx} - 2i|u_x|^2u - 3iu_x^2\overline{u} - \frac{15}{2}|u|^4u_x + \frac{5i}{2}|u|^6u)
    \end{equation*}
    cannot be thrice continuously differentiable for any $1 \le r \le \infty$ and $s \in\R$.
\end{proposition}
\begin{proof}
The proof of this proposition works similarly to the one given by the author in~\cite{Adams2024}*{Proposition~6.3}, which in turn was based on an argument by Bourgain~\cite{Bourgain1997}.

In the present setting we may observe, that the symbol of the cubic nonlinearity in the fourth-order hierarchy equation can be written as
\begin{equation}\label{eq:fourth-order-symbol}
    n_3(k_1, k_2, k_3) = (k_1 + k_2 + k_3)(2k_1^2 + k_2^2 + 2k_3^2 + k_1k_2 + k_2k_3 + 3k_1k_3).
\end{equation}
Following along the details of~\cite{Adams2024}*{Proposition~6.3}, i.e. differentiating the flow thrice (with respect to $\delta$) with initial data $\delta\phi(x)$, where $\delta > 0$ and $\hat{\phi}(k) = k^{-s}(\delta_{k, N} + \delta_{k, N_0})$ and looking for a lower bound on the $H^s(\T)$ norm of this third derivative, one arrives at the same conclusion. Only for $(N, N_0, N_0)$ and $(N, -N, -N)$ (or appropriate permutations thereof) an overall frequency of $N$ is achieved. Inserting these constellations into~\eqref{eq:fourth-order-symbol} one may derive a lower bound of $N^s t N^{3-s}(1 + N^{-2s}) \gtrsim t N^3$ for the $H^s(\T)$ norm of the derivative. For $N \to \infty$ this diverges, so we know the flow cannot be thrice continuously differentiable. We leave working out further details to the reader.
\end{proof}

If one lowers the assumption on the regularity of the initial data, one is able to strengthen the form of ill-posedness that is derived to failure of uniform continuity using an argument originally developed in~\cite{KPV2001}.

\begin{proposition}\label{prop:illposed-torus-uniform}
    Let $j \in \N$, $1 \le r \le \infty$ and $s < j -\1{2}$. The flow $S : \hat H^s_r(\T) \times (-T, T) \to \hat H^s_r(\T)$ of the Cauchy problem
    \begin{equation}\label{eq:torus-illposed-pde}
        i\partial_t u + (-1)^{j+1}\partial_x^{2j} u = i u^2 \partial_x^{2j-1}\overline{u}
    \end{equation}
    cannot be uniformly continuous on bounded sets.
\end{proposition}
\begin{proof}
The proof of this proposition follows the same argument already given by the author for~\cite{Adams2024}*{Proposition~6.4}, so we will not repeat the details here. The only difference is that one must choose a different particular solution of~\eqref{eq:torus-illposed-pde}, which is a slightly modified (i.e. adapted to the dNLS hierarchy setting) version of~\cite{Adams2024}*{eq.~(6.5)}.
In particular it suffices to use the family of solutions
\begin{equation}
    u_{N, a}(x, t) = N^{-s} a \exp(i(Nx - N^{2j}t + N^{2j-1-2s}|a|^2 t))
\end{equation}
in this case.

Note that as all derivatives in this equation fall on $\overline{u}$ this is in fact a gauged dNLS equation. Changing the sign in front of the nonlinearity allows one to solve (using the same family $u_{N, a}$) the equation where all derivatives fall on $u$ instead.
\end{proof}

    \appendix
    \section{The first few dNLS hierarchy equations}\label{sec:appendix}

For future reference and the interested reader we would like to list the first few equations of the dNLS hierarchy and the resulting equations after they have been gauge transformed for the Schrödinger-like ones. A similar listing concerning the NLS hierarchy equations may be found in~\cite{Adams2024}*{Appendix~A}.

We will give the equations in terms of the potentials $q$ and $r$ as in the description of the hierarchy with a specific choice of $\alpha_n$ left to the reader (except for all other $\alpha_n$ being zero), as in Section~\ref{sec:dnls-derivation}. The usual identification $r = \pm \overline{q}$ leads to the well-known equations found elsewhere in the literature. For the non-gauge transformed equations we give them once with the nonlinearity as a total derivative, as in the representation \eqref{eq:dnls-hierarchy}, and again but with the derivative applied.

For the equations which have been adjusted with the gauge transform we use the convention $\alpha_{2j-1}=\alpha_n=2^{2j-1}$ which has been in use throughout the rest of the text as well.
\pagebreak
\begin{enumerate}[wide]
\item $n=0$. transport equation $q_t = \alpha_0 q_x$

\item $n=1$. classic dNLS equation
\begin{align*}
q_t = \frac{i\alpha_1}{2}(q_{xx} + \partial_x(-iq^2r)) = \frac{i\alpha_1}{2}(q_{xx} -2i qq_xr -iq^2r_x)
\end{align*}
After gauge transformation:
\begin{align*}
    iq_t + q_{xx} = -iq^2r_x -\frac{1}{2}q^3r^2
\end{align*}

\item $n=2$.
\begin{align*}
	q_t & = -\frac{\alpha_2}{4}(q_{xxx} + \partial_x(-3i qq_xr -\frac{3}{2} q^3r^2))                         \\
	    & = -\frac{\alpha_2}{4}(q_{xxx} -3iq_x^2r -3iqq_xr_x -3iqq_{xx}r -3q^3rr_x - \frac{9}{2}q^2q_xr^2)
\end{align*}

\item $n=3$. fourth order dNLS equation
\begin{align*}
	q_t & = -\frac{i\alpha_3}{8}(q_{xxxx} +  \partial_x(-iq^2r_{xx} -4i qq_{xx}r -2i qq_xr_x - 3i q_x^2r        \\
	    &                                    -\frac{15}{2}q^2q_xr^2 + \frac{5i}{2} q^4r^3))                     \\
	    & = -\frac{i\alpha_3}{8}(q_{xxxx}    -iq^2r_{xxx} -4iqq_{xxx}r -4iqq_xr_{xx} -6iqq_{xx}r_x              \\
	    &                                    -10iq_xq_{xx}r -5iq_x^2r_x -\frac{15}{2}q^2q_{xx}r^2 -15q^2q_xrr_x \\
	    &                                    -15qq_x^2r^2 +\frac{15i}{2}q^4r2r_x + 10iq^3q_xr^3)
\end{align*}
After gauge transformation:
\begin{align*}
	iq_t - q_{xxxx} & = iq^2r_{xxx} +2iqq_xr_{xx} +4iqq_{xx}r_x +3iq_x^2r_x +q^3rr_{xx}                                                          \\
	                & +\frac{5}{2}q^2q_{xx}r^2 -\frac{1}{2}q^3r_x^2 +4q^2q_xrr_x +\frac{5}{2}qq_x^2r^2 +\frac{3i}{2}q^4r^2r_x +\frac{3}{8}q^5r^4
\end{align*}

\item $n=4$.
\begin{align*}
	q_t & = \frac{\alpha_4}{16}(q_{xxxxx} +\partial_x(-5iqq_{xxx}r -5iqq_xr_{xx} -5iqq_{xx}r_x -10iq_xq_{xx}r                               \\
	    & -5iq_x^2r_x -5q^3rr_{xx} -\frac{25}{2}q^2q_{xx}r^2 -\frac{5}{2}q^3r_x^2 -15 q^2q_xrr_x -\frac{35}{2}qq_x^2r^2                     \\
	    & +\frac{35i}{2}q^3q_xr^3 +\frac{35}{8}q^5r^4))                                                                                     \\
	    & = \frac{\alpha_4}{16}(q_{xxxxx} -5i qq_{xxxx}r -5iqq_xr_{xxx} -10iqq_{xxx}r_x -10iqq_{xx}r_{xx} -10iq_x^2r_{xx}                   \\
	    & -10iq_{xx}^2r -25iq_xq_{xx}r_{x} -15q_xq_{xxx}r -5q^3rr_{xxx} -\frac{25}{2}q^2q_{xxx}r^2 -10q^3r_xr_{xx}                          \\
	    & -30q^2q_xrr_{xx} -40q^2q_{xx}rr_x -60qq_xq_{xx}r^2 -\frac{45}{2}q^2q_xr_x^2 -65qq_x^2rr_x -\frac{35}{2}q_x^3r^2                   \\
	    & +\frac{35i}{2}q^3q_{xx}r^3 +\frac{105i}{2}q^3q_xr^2r_x +\frac{35}{2}q^5r^3r_x +\frac{105i}{2}q^2q_x^2r^3 +\frac{175}{8}q^4q_xr^4)
\end{align*}

\item $n=5$. sixth order dNLS equation
\begin{align*}
	q_t & = \frac{i\alpha_5}{32}(q_{xxxxxx} +  \partial_x(-iq^2r_{xxxx} -6iqq_{xxxx}r -4iqq_xr_{xxx} -9iqq_{xxx}r_x                   \\
	    & -15iq_xq_{xxx}r -11iqq_{xx}r_{xx} -10iq_x^2r_{xx} -10iq_{xx}^2r -25q_xq_xxr_x -\frac{35}{2}q^2q_{xxx}r^2                    \\
	    & -35q^2q_xrr_{xx} -35q^2q_{xx}rr_x -70qq_xq_{xx}r^2 -\frac{35}{2}q^2q_xr_x^2 -70qq_x^2rr_x -\frac{35}{2}q_x^3r^2             \\
	    & +\frac{35i}{2}q^4r^2r_{xx} +35iq^3q_{xx}r^3 +\frac{35i}{2}q^4rr_x^2 +70iq^3q_xr^2r_x +70iq^2q_x^2r^3 \frac{315}{8}q^4q_xr^4 \\
	    & -\frac{63i}{8}q^6r^5))                                                                                                      \\
	    & = \frac{i\alpha_5}{32}(q_{xxxxxx} -iq^2r_{xxxxx} -6iqq_{xxxxx}r -6iqq_xr_{xxxx} -15iqq_{xxxx}r_x                            \\
	    & -21iq_xq_{xxxx}r -15iqq_{xx}r_{xxx} -14iq_x^2r_{xxx} -20iqq_{xxx}r_{xx} -35iq_{xx}q_{xxx}r                                   \\
	    & -49iq_xq_{xxx}r_x -56iq_xq_{xx}r_{xx} -35iq_{xx}^2r_x -\frac{35}{2}q^2q_{xxxx}r^2 -35q^2q_xr_{xxx}r                         \\
	    & -70q^2q_{xxx}rr_x -105qq_xq_{xxx}r^2 -70q^2q_{xx}rr_{xx} -70q^2q_xr_xr_{xx} -140qq_x^2rr_{xx}                               \\
	    & -70qq_{xx}^2r^2 -\frac{105}{2}q^2q_{xx}r_x^2 -350qq_xq_{xx}rr_x -\frac{245}{2}q_x^2q_{xx}r^2 -105qq_x^2r_x^2                \\
	    & -105q_x^3rr_x +\frac{35i}{2}q^4r^2r_{xxx} +35iq^3q_{xxx}r^3 +70iq^4rr_xr_{xx} +140iq^3q_xr^2r_{xx}                          \\
	    & +175iq^3q_{xx}r^2r_x +245iq^2q_xq_{xx}r^3 +\frac{35i}{2}q^4r_x^3 +210iq^3q_xrr_x^2 +420iq^2q_x^2r^2r_x                      \\
	    & +140iqq_x^3r^3 +\frac{315}{8}q^4q_{xx}r^4 +\frac{315}{2}q^4q_xr^3r_x \frac{315}{2}q^3q_x^2r^4 -\frac{315i}{8}q^6r^4r_x      \\
	    & -\frac{189i}{4}q^5q_xr^5)
\end{align*}
After gauge transformation:
\begin{align*}
iq_t + \partial_x^6 q_{} & =
-iq^2r_{xxxxx}
-4iqq_xr_{xxxx}
-6iqq_{xxxx}r_x
-11iqq_{xx}r_{xxx}
-10iq_x^2r_{xxx}\\&
-9iqq_{xxx}r_{xx}
-15iq_xq_{xxx}r_x
-25iq_xq_{xx}r_{xx}
-10iq_{xx}^2r_x
-q^3rr_{xxxx}\\&
-\frac{7}{2}q^2q_{xxxx}r^2
+q^3r_xr_{xxx}
-8q^2q_xrr_{xxx}
-13q^2q_{xxx}rr_x
-14qq_xq_{xxx}r^2\\&
-\frac{1}{2}q^3r_{xx}^2
-17q^2q_{xx}rr_{xx}
-9q^2q_xr_xr_{xx}
-22qq_x^2rr_{xx}
-\frac{21}{2}qq_{xx}^2r^2\\&
-\frac{9}{2}q^2q_{xx}r_x^2
-59qq_xq_{xx}rr_x
-\frac{35}{2}q_x^2q_{xx}r^2
-\frac{9}{2}qq_x^2r_x^2
-20q_x^3rr_x\\&
-\frac{5i}{2}q^4r^2r_{xxx}
-10iq^4rr_xr_{xx}
-10iq^3q_xr^2r_{xx}
-15iq^3q_{xx}r^2r_x
-\frac{5i}{2}q^4r_x^3\\&
-25iq^3y_xrr_x^2
-25iq^2q_x^2r^2r_x
-\frac{5}{2}q^5r^3r_{xx}
-\frac{35}{8}q^4q_{xx}r^4
-\frac{5}{4}q^5r^2r_x^2\\&
-15q^4q_xr^3r_x
-\frac{35}{4}q^3q_x^2r^4
-\frac{5}{16}q^7r^6
-\frac{15i}{8}q^6r^4r_x
\end{align*}
Note the sign difference of the term $+q^3r_xr_{xxx}$ to all others with four derivatives lying upon them in the gauge transformed equation. This does not seem to be a mistake originating from the derivation of the equation.
\end{enumerate}

    \pagebreak
	\bibliography{bibliography.bib}

\begin{bibdiv}
\begin{biblist}

\bib{AKNS1974}{article}{
      author={Ablowitz, Mark~J.},
      author={Kaup, David~J.},
      author={Newell, Alan~C.},
      author={Segur, Harvey},
       title={The inverse scattering transform-{F}ourier analysis for nonlinear
  problems},
        date={1974},
        ISSN={0022-2526,1467-9590},
     journal={Studies in Appl. Math.},
      volume={53},
      number={4},
       pages={249\ndash 315},
         url={https://doi.org/10.1002/sapm1974534249},
      review={\MR{450815}},
}

\bib{Adams2024}{article}{
      author={Adams, Joseph},
       title={{Well-posedness for the NLS hierarchy}},
        date={2024},
     journal={J. Evol. Equ.},
      volume={24},
      number={88},
}

\bib{agrawal2000nonlinear}{incollection}{
      author={Agrawal, Govind~P},
       title={Nonlinear fiber optics},
        date={2000},
   booktitle={Nonlinear {S}cience at the {D}awn of the 21st {C}entury},
   publisher={Springer},
       pages={195\ndash 211},
}

\bib{Alberty1982-i}{article}{
      author={Alberty, J.~M.},
      author={Koikawa, T.},
      author={Sasaki, R.},
       title={Canonical structure of soliton equations. {I}},
        date={1982},
        ISSN={0167-2789,1872-8022},
     journal={Phys. D},
      volume={5},
      number={1},
       pages={43\ndash 65},
         url={https://doi.org/10.1016/0167-2789(82)90049-5},
      review={\MR{666528}},
}

\bib{PhysRevA.27.1393}{article}{
      author={Anderson, D.},
      author={Lisak, M.},
       title={Nonlinear asymmetric self-phase modulation and self-steepening of
  pulses in long optical waveguides},
        date={1983},
     journal={Phys. Rev. A},
      volume={27},
       pages={1393\ndash 1398},
         url={https://link.aps.org/doi/10.1103/PhysRevA.27.1393},
}

\bib{Benyi2020}{book}{
      author={B\'{e}nyi, \'{A}rp\'{a}d},
      author={Okoudjou, Kasso~A.},
       title={Modulation spaces},
      series={Applied and Numerical Harmonic Analysis},
   publisher={Birkh\"{a}user/Springer, New York},
        date={2020},
        ISBN={978-1-0716-0330-7; 978-1-0716-0332-1},
         url={https://doi.org/10.1007/978-1-0716-0332-1},
      review={\MR{4286055}},
}

\bib{BiagioniLinares2001}{article}{
      author={Biagioni, H.~A.},
      author={Linares, F.},
       title={Ill-posedness for the derivative {S}chr\"odinger and generalized
  {B}enjamin-{O}no equations},
        date={2001},
        ISSN={0002-9947,1088-6850},
     journal={Trans. Amer. Math. Soc.},
      volume={353},
      number={9},
       pages={3649\ndash 3659},
         url={https://doi.org/10.1090/S0002-9947-01-02754-4},
      review={\MR{1837253}},
}

\bib{Bourgain1993-1}{article}{
      author={Bourgain, J.},
       title={Fourier transform restriction phenomena for certain lattice
  subsets and applications to nonlinear evolution equations. {I}.
  {S}chr\"{o}dinger equations},
        date={1993},
        ISSN={1016-443X,1420-8970},
     journal={Geom. Funct. Anal.},
      volume={3},
      number={2},
       pages={107\ndash 156},
         url={https://doi.org/10.1007/BF01896020},
      review={\MR{1209299}},
}

\bib{Bourgain1993-2}{article}{
      author={Bourgain, J.},
       title={Fourier transform restriction phenomena for certain lattice
  subsets and applications to nonlinear evolution equations. {II}. {T}he
  {K}d{V}-equation},
        date={1993},
        ISSN={1016-443X,1420-8970},
     journal={Geom. Funct. Anal.},
      volume={3},
      number={3},
       pages={209\ndash 262},
         url={https://doi.org/10.1007/BF01895688},
      review={\MR{1215780}},
}

\bib{Bourgain1997}{article}{
      author={Bourgain, J.},
       title={Periodic {K}orteweg de {V}ries equation with measures as initial
  data},
        date={1997},
        ISSN={1022-1824,1420-9020},
     journal={Selecta Math. (N.S.)},
      volume={3},
      number={2},
       pages={115\ndash 159},
         url={https://doi.org/10.1007/s000290050008},
      review={\MR{1466164}},
}

\bib{Chaichenets2018}{thesis}{
      author={Chaichenets, Leonid},
       title={Modulation spaces and nonlinear {S}chrödinger equations},
        type={Ph.D. Thesis},
        date={2018},
}

\bib{MR1871414}{article}{
      author={Colliander, J.},
      author={Keel, M.},
      author={Staffilani, G.},
      author={Takaoka, H.},
      author={Tao, T.},
       title={Global well-posedness for {S}chr\"odinger equations with
  derivative},
        date={2001},
        ISSN={0036-1410,1095-7154},
     journal={SIAM J. Math. Anal.},
      volume={33},
      number={3},
       pages={649\ndash 669},
         url={https://doi.org/10.1137/S0036141001384387},
      review={\MR{1871414}},
}

\bib{MR1950826}{article}{
      author={Colliander, J.},
      author={Keel, M.},
      author={Staffilani, G.},
      author={Takaoka, H.},
      author={Tao, T.},
       title={A refined global well-posedness result for {S}chr\"odinger
  equations with derivative},
        date={2002},
        ISSN={0036-1410,1095-7154},
     journal={SIAM J. Math. Anal.},
      volume={34},
      number={1},
       pages={64\ndash 86},
         url={https://doi.org/10.1137/S0036141001394541},
      review={\MR{1950826}},
}

\bib{MR4259382}{article}{
      author={Deng, Yu},
      author={Nahmod, Andrea~R.},
      author={Yue, Haitian},
       title={Optimal local well-posedness for the periodic derivative
  nonlinear {S}chr\"odinger equation},
        date={2021},
        ISSN={0010-3616,1432-0916},
     journal={Comm. Math. Phys.},
      volume={384},
      number={2},
       pages={1061\ndash 1107},
         url={https://doi.org/10.1007/s00220-020-03898-8},
      review={\MR{4259382}},
}

\bib{Faddeev1987}{book}{
      author={Faddeev, L.~D.},
      author={Takhtajan, L.~A.},
       title={Hamiltonian methods in the theory of solitons},
      series={Springer Series in Soviet Mathematics},
   publisher={Springer-Verlag, Berlin},
        date={1987},
        ISBN={3-540-15579-1},
         url={https://doi.org/10.1007/978-3-540-69969-9},
        note={Translated from the Russian by A. G. Reyman [A. G. Re\u{\i}man]},
      review={\MR{905674}},
}

\bib{GagliardoNirenberg}{article}{
      author={Fiorenza, Alberto},
      author={Formica, Maria~Rosaria},
      author={Roskovec, T.~G.},
      author={Soudsk\'y, Filip},
       title={Detailed proof of classical {G}agliardo-{N}irenberg interpolation
  inequality with historical remarks},
        date={2021},
        ISSN={0232-2064,1661-4534},
     journal={Z. Anal. Anwend.},
      volume={40},
      number={2},
       pages={217\ndash 236},
         url={https://doi.org/10.4171/zaa/1681},
      review={\MR{4237368}},
}

\bib{Forlano2019}{article}{
      author={Forlano, Justin},
      author={Trenberth, William~J.},
       title={On the transport of {G}aussian measures under the one-dimensional
  fractional nonlinear {S}chr\"{o}dinger equations},
        date={2019},
        ISSN={0294-1449,1873-1430},
     journal={Ann. Inst. H. Poincar\'{e} C Anal. Non Lin\'{e}aire},
      volume={36},
      number={7},
       pages={1987\ndash 2025},
         url={https://doi.org/10.1016/j.anihpc.2019.07.006},
      review={\MR{4020531}},
}

\bib{AGDiss}{thesis}{
      author={Gr\"{u}nrock, Axel},
       title={{New applications of the Fourier restriction norm method to
  wellposedness problems for nonlinear evolution equations}},
        type={Ph.D. Thesis},
        date={2002},
}

\bib{Grünrock2004}{article}{
      author={Gr\"{u}nrock, Axel},
       title={An improved local well-posedness result for the modified {K}d{V}
  equation},
        date={2004},
        ISSN={1073-7928,1687-0247},
     journal={Int. Math. Res. Not.},
      number={61},
       pages={3287\ndash 3308},
         url={https://doi.org/10.1155/S1073792804140981},
      review={\MR{2096258}},
}

\bib{AG2005}{article}{
      author={Gr\"{u}nrock, Axel},
       title={Bi- and trilinear {S}chr\"{o}dinger estimates in one space
  dimension with applications to cubic {NLS} and {DNLS}},
        date={2005},
        ISSN={1073-7928,1687-0247},
     journal={Int. Math. Res. Not.},
      number={41},
       pages={2525\ndash 2558},
         url={https://doi.org/10.1155/IMRN.2005.2525},
      review={\MR{2181058}},
}

\bib{AGTowers}{article}{
      author={Gr\"{u}nrock, Axel},
       title={On the hierarchies of higher-order m{K}d{V} and {K}d{V}
  equations},
        date={2010},
        ISSN={1895-1074,1644-3616},
     journal={Cent. Eur. J. Math.},
      volume={8},
      number={3},
       pages={500\ndash 536},
         url={https://doi.org/10.2478/s11533-010-0024-5},
      review={\MR{2653659}},
}

\bib{MR2390318}{article}{
      author={Gr\"unrock, Axel},
      author={Herr, Sebastian},
       title={Low regularity local well-posedness of the derivative nonlinear
  {S}chr\"odinger equation with periodic initial data},
        date={2008},
        ISSN={0036-1410,1095-7154},
     journal={SIAM J. Math. Anal.},
      volume={39},
      number={6},
       pages={1890\ndash 1920},
         url={https://doi.org/10.1137/070689139},
      review={\MR{2390318}},
}

\bib{Guo2021}{article}{
      author={Guo, Shaoming},
      author={Ren, Xianfeng},
      author={Wang, Baoxiang},
       title={Local well-posedness for the derivative nonlinear
  {S}chr\"{o}dinger equation with {$L^2$}-subcritical data},
        date={2021},
        ISSN={1078-0947,1553-5231},
     journal={Discrete Contin. Dyn. Syst.},
      volume={41},
      number={9},
       pages={4207\ndash 4253},
         url={https://doi.org/10.3934/dcds.2021034},
      review={\MR{4256461}},
}

\bib{MR3583477}{article}{
      author={Guo, Zihua},
      author={Wu, Yifei},
       title={Global well-posedness for the derivative nonlinear
  {S}chr\"odinger equation in {$H^{\frac{1}{2}}(\Bbb{R})$}},
        date={2017},
        ISSN={1078-0947,1553-5231},
     journal={Discrete Contin. Dyn. Syst.},
      volume={37},
      number={1},
       pages={257\ndash 264},
         url={https://doi.org/10.3934/dcds.2017010},
      review={\MR{3583477}},
}

\bib{MR4565673}{article}{
      author={Harrop-Griffiths, Benjamin},
      author={Killip, Rowan},
      author={Vi\c{s}an, Monica},
       title={Large-data equicontinuity for the derivative {NLS}},
        date={2023},
        ISSN={1073-7928,1687-0247},
     journal={Int. Math. Res. Not. IMRN},
      number={6},
       pages={4601\ndash 4642},
         url={https://doi.org/10.1093/imrn/rnab374},
      review={\MR{4565673}},
}

\bib{HayashiOzawa1992}{article}{
      author={Hayashi, Nakao},
      author={Ozawa, Tohru},
       title={On the derivative nonlinear {S}chr\"odinger equation},
        date={1992},
        ISSN={0167-2789,1872-8022},
     journal={Phys. D},
      volume={55},
      number={1-2},
       pages={14\ndash 36},
         url={https://doi.org/10.1016/0167-2789(92)90185-P},
      review={\MR{1152001}},
}

\bib{Herr2006}{article}{
      author={Herr, Sebastian},
       title={On the {C}auchy problem for the derivative nonlinear
  {S}chr\"odinger equation with periodic boundary condition},
        date={2006},
        ISSN={1073-7928,1687-0247},
     journal={Int. Math. Res. Not.},
       pages={Art. ID 96763, 33},
         url={https://doi.org/10.1155/IMRN/2006/96763},
      review={\MR{2219223}},
}

\bib{Ikeda2021}{article}{
      author={Hirayama, Hiroyuki},
      author={Ikeda, Masahiro},
      author={Tanaka, Tomoyuki},
       title={Well-posedness for the fourth-order {S}chr\"{o}dinger equation
  with third order derivative nonlinearities},
        date={2021},
        ISSN={1021-9722,1420-9004},
     journal={NoDEA Nonlinear Differential Equations Appl.},
      volume={28},
      number={5},
       pages={Paper No. 46, 72},
         url={https://doi.org/10.1007/s00030-021-00707-6},
      review={\MR{4274694}},
}

\bib{MR2818712}{article}{
      author={Huo, Zhaohui},
      author={Jia, Yueling},
       title={Well-posedness for the fourth-order nonlinear derivative
  {S}chr\"odinger equation in higher dimension},
        date={2011},
        ISSN={0021-7824,1776-3371},
     journal={J. Math. Pures Appl. (9)},
      volume={96},
      number={2},
       pages={190\ndash 206},
         url={https://doi.org/10.1016/j.matpur.2011.01.002},
      review={\MR{2818712}},
}

\bib{MR1230709}{article}{
      author={Kenig, Carlos~E.},
      author={Ponce, Gustavo},
      author={Vega, Luis},
       title={Small solutions to nonlinear {S}chr\"odinger equations},
        date={1993},
        ISSN={0294-1449,1873-1430},
     journal={Ann. Inst. H. Poincar\'e{} C Anal. Non Lin\'eaire},
      volume={10},
      number={3},
       pages={255\ndash 288},
         url={https://doi.org/10.1016/S0294-1449(16)30213-X},
      review={\MR{1230709}},
}

\bib{KPV2001}{article}{
      author={Kenig, Carlos~E.},
      author={Ponce, Gustavo},
      author={Vega, Luis},
       title={On the ill-posedness of some canonical dispersive equations},
        date={2001},
        ISSN={0012-7094,1547-7398},
     journal={Duke Math. J.},
      volume={106},
      number={3},
       pages={617\ndash 633},
         url={https://doi.org/10.1215/S0012-7094-01-10638-8},
      review={\MR{1813239}},
}

\bib{MR4628747}{article}{
      author={Killip, Rowan},
      author={Ntekoume, Maria},
      author={Vi\c{s}an, Monica},
       title={On the well-posedness problem for the derivative nonlinear
  {S}chr\"odinger equation},
        date={2023},
        ISSN={2157-5045,1948-206X},
     journal={Anal. PDE},
      volume={16},
      number={5},
       pages={1245\ndash 1270},
         url={https://doi.org/10.2140/apde.2023.16.1245},
      review={\MR{4628747}},
}

\bib{Klaus2023}{article}{
      author={Klaus, Friedrich},
       title={Wellposedness of {NLS} in modulation spaces},
        date={2023},
        ISSN={1069-5869,1531-5851},
     journal={J. Fourier Anal. Appl.},
      volume={29},
      number={1},
       pages={Paper No. 9, 37},
         url={https://doi.org/10.1007/s00041-022-09985-9},
      review={\MR{4534495}},
}

\bib{KochKlaus2023}{article}{
      author={Klaus, Friedrich},
      author={Koch, Herbert},
      author={Liu, Baoping},
       title={{Wellposedness for the KdV hierarchy}},
        date={2023},
      eprint={arXiv:2309.12773},
}

\bib{Koch2018}{article}{
      author={Koch, Herbert},
      author={Tataru, Daniel},
       title={Conserved energies for the cubic nonlinear {S}chr\"{o}dinger
  equation in one dimension},
        date={2018},
        ISSN={0012-7094,1547-7398},
     journal={Duke Math. J.},
      volume={167},
      number={17},
       pages={3207\ndash 3313},
         url={https://doi.org/10.1215/00127094-2018-0033},
      review={\MR{3874652}},
}

\bib{MR2823664}{article}{
      author={Miao, Changxing},
      author={Wu, Yifei},
      author={Xu, Guixiang},
       title={Global well-posedness for {S}chr\"odinger equation with
  derivative in {$H^{\frac12}(\Bbb R)$}},
        date={2011},
        ISSN={0022-0396,1090-2732},
     journal={J. Differential Equations},
      volume={251},
      number={8},
       pages={2164\ndash 2195},
         url={https://doi.org/10.1016/j.jde.2011.07.004},
      review={\MR{2823664}},
}

\bib{MR462141}{article}{
      author={Mio, Koji},
      author={Ogino, Tatsuki},
      author={Minami, Kazuo},
      author={Takeda, Susumu},
       title={Modified nonlinear {S}chr\"odinger equation for {A}lfv\'en waves
  propagating along the magnetic field in cold plasmas},
        date={1976},
        ISSN={0031-9015,1347-4073},
     journal={J. Phys. Soc. Japan},
      volume={41},
      number={1},
       pages={265\ndash 271},
         url={https://doi.org/10.1143/JPSJ.41.265},
      review={\MR{462141}},
}

\bib{Palais1997}{article}{
      author={Palais, Richard~S.},
       title={The symmetries of solitons},
        date={1997},
        ISSN={0273-0979,1088-9485},
     journal={Bull. Amer. Math. Soc. (N.S.)},
      volume={34},
      number={4},
       pages={339\ndash 403},
         url={https://doi.org/10.1090/S0273-0979-97-00732-5},
      review={\MR{1462745}},
}

\bib{Sasaki1982-ii}{article}{
      author={Sasaki, R.},
       title={Canonical structure of soliton equations. {II}. {T}he
  {K}aup-{N}ewell system},
        date={1982},
        ISSN={0167-2789,1872-8022},
     journal={Phys. D},
      volume={5},
      number={1},
       pages={66\ndash 74},
         url={https://doi.org/10.1016/0167-2789(82)90050-1},
      review={\MR{666529}},
}

\bib{Takaoka1999}{article}{
      author={Takaoka, Hideo},
       title={Well-posedness for the one-dimensional nonlinear {S}chr\"odinger
  equation with the derivative nonlinearity},
        date={1999},
        ISSN={1079-9389},
     journal={Adv. Differential Equations},
      volume={4},
      number={4},
       pages={561\ndash 580},
      review={\MR{1693278}},
}

\bib{Takaoka2001}{article}{
      author={Takaoka, Hideo},
       title={Global well-posedness for {S}chr\"odinger equations with
  derivative in a nonlinear term and data in low-order {S}obolev spaces},
        date={2001},
        ISSN={1072-6691},
     journal={Electron. J. Differential Equations},
       pages={No. 42, 23},
      review={\MR{1836810}},
}

\bib{MR621533}{article}{
      author={Tsutsumi, Masayoshi},
      author={Fukuda, Isamu},
       title={On solutions of the derivative nonlinear {S}chr\"odinger
  equation. {E}xistence and uniqueness theorem},
        date={1980},
        ISSN={0532-8721},
     journal={Funkcial. Ekvac.},
      volume={23},
      number={3},
       pages={259\ndash 277},
         url={http://www.math.kobe-u.ac.jp/~fe/xml/mr0621533.xml},
      review={\MR{621533}},
}

\bib{MR634894}{article}{
      author={Tsutsumi, Masayoshi},
      author={Fukuda, Isamu},
       title={On solutions of the derivative nonlinear {S}chr\"odinger
  equation. {II}},
        date={1981},
        ISSN={0532-8721},
     journal={Funkcial. Ekvac.},
      volume={24},
      number={1},
       pages={85\ndash 94},
         url={http://www.math.kobe-u.ac.jp/~fe/xml/mr0634894.xml},
      review={\MR{634894}},
}

\bib{PhysRevA.23.1266}{article}{
      author={Tzoar, N.},
      author={Jain, M.},
       title={Self-phase modulation in long-geometry optical waveguides},
        date={1981},
     journal={Phys. Rev. A},
      volume={23},
       pages={1266\ndash 1270},
         url={https://link.aps.org/doi/10.1103/PhysRevA.23.1266},
}

\bib{MR1735881}{article}{
      author={Tzvetkov, Nickolay},
       title={Remark on the local ill-posedness for {K}d{V} equation},
        date={1999},
        ISSN={0764-4442},
     journal={C. R. Acad. Sci. Paris S\'er. I Math.},
      volume={329},
      number={12},
       pages={1043\ndash 1047},
         url={https://doi.org/10.1016/S0764-4442(00)88471-2},
      review={\MR{1735881}},
}

\bib{MR3198590}{article}{
      author={Wu, Yifei},
       title={Global well-posedness for the nonlinear {S}chr\"odinger equation
  with derivative in energy space},
        date={2013},
        ISSN={2157-5045,1948-206X},
     journal={Anal. PDE},
      volume={6},
      number={8},
       pages={1989\ndash 2002},
         url={https://doi.org/10.2140/apde.2013.6.1989},
      review={\MR{3198590}},
}

\bib{MR3393674}{article}{
      author={Wu, Yifei},
       title={Global well-posedness on the derivative nonlinear {S}chr\"odinger
  equation},
        date={2015},
        ISSN={2157-5045,1948-206X},
     journal={Anal. PDE},
      volume={8},
      number={5},
       pages={1101\ndash 1112},
         url={https://doi.org/10.2140/apde.2015.8.1101},
      review={\MR{3393674}},
}

\bib{MR4348344}{article}{
      author={Zhou, Huaxin},
      author={Yu, Jing},
      author={Han, Jingwei},
       title={Several exact solutions of the reduced fourth-order flow equation
  of the {K}aup-{N}ewell system},
        date={2022},
        ISSN={0165-2125,1878-433X},
     journal={Wave Motion},
      volume={109},
       pages={Paper No. 102840, 12},
         url={https://doi.org/10.1016/j.wavemoti.2021.102840},
      review={\MR{4348344}},
}

\end{biblist}
\end{bibdiv}
\end{document}